\documentclass{article}

\usepackage{arxiv}

\usepackage[utf8]{inputenc}
\usepackage[T1]{fontenc}
\usepackage{amsmath}
\usepackage{amsfonts}
\usepackage{amssymb}
\usepackage{amsthm}
\usepackage{mathtools}
\usepackage{tikz}
\usepackage{pgfplots}
\usepackage{algorithm}%
\usepackage{algorithmicx}%
\usepackage[noend]{algpseudocode}
\usepackage[colorlinks=true, linkcolor=blue]{hyperref}
\usepackage{cleveref}
\usepackage{booktabs}
\usepackage{mdframed}
\usepackage{appendix}

\usepgfplotslibrary{external}
\tikzset{external/system call = {%
    pdflatex \tikzexternalcheckshellescape%
    -halt-on-error
    -interaction=batchmode
    -jobname "\image" "\texsource"}}
\newtheorem{theorem}{Theorem}[section]

\newtheorem{lemma}[theorem]{Lemma}
\newtheorem{corollary}[theorem]{Corollary}
\newtheorem{example}[theorem]{Example}
\newtheorem{remark}[theorem]{Remark}
\newtheorem{definition}[theorem]{Definition}
\newtheorem{assumption}[theorem]{Assumption}

\algrenewcommand\algorithmicrequire{\textbf{Input:}}
\algrenewcommand\algorithmicensure{\textbf{Output:}}

\numberwithin{equation}{section}

\usepackage{amsmath}
\usepackage{amsfonts}
\usepackage{amssymb}
\usepackage{mathtools}
\usepackage{tikz}
\usepackage{pgfplots}

\usepackage{acronym}
\acrodef{fom}[FOM]{full-order model}
\acrodef{lti}[LTI]{linear time-invariant}
\acrodef{mse}[MSE]{mean square error}
\acrodef{pde}[PDE]{partial differential equation}
\acrodef{pod}[POD]{proper orthogonal decomposition}
\acrodef{rb}[RB]{reduced basis method}
\acrodef{rom}[ROM]{reduced-order model}
\acrodef{ddrom}[DDROM]{data-driven reduced-order model}
\acrodefindefinite{lti}{an}{a}
\acrodefindefinite{mse}{an}{a}

\newcommand{\nfom}{\ensuremath{n}}
\newcommand{\nin}{\ensuremath{n_{\textnormal{f}}}}
\newcommand{\nout}{\ensuremath{n_{\textnormal{o}}}}

\newcommand{\Af}{\ensuremath{\mathcal{A}}}
\newcommand{\Bf}{\ensuremath{\mathcal{B}}}
\newcommand{\Cf}{\ensuremath{\mathcal{C}}}

\newcommand{\cAf}{\ensuremath{A}}
\newcommand{\cBf}{\ensuremath{B}}
\newcommand{\cCf}{\ensuremath{C}}

\newcommand{\caf}{\ensuremath{\alpha}}
\newcommand{\cbf}{\ensuremath{\beta}}
\newcommand{\ccf}{\ensuremath{\gamma}}

\newcommand{\qAf}{\ensuremath{q_{\Af}}}
\newcommand{\qBf}{\ensuremath{q_{\Bf}}}
\newcommand{\qCf}{\ensuremath{q_{\Cf}}}

\newcommand{\xf}{\ensuremath{x}}
\newcommand{\yf}{\ensuremath{y}}

\newcommand{\nrom}{\ensuremath{r}}

\newcommand{\Ar}{\ensuremath{\hat{\mathcal{A}}}}
\newcommand{\Br}{\ensuremath{\hat{\mathcal{B}}}}
\newcommand{\Cr}{\ensuremath{\hat{\mathcal{C}}}}

\newcommand{\cAr}{\ensuremath{\hat{A}}}
\newcommand{\cBr}{\ensuremath{\hat{B}}}
\newcommand{\cCr}{\ensuremath{\hat{C}}}

\newcommand{\car}{\ensuremath{\hat{\alpha}}}
\newcommand{\cbr}{\ensuremath{\hat{\beta}}}
\newcommand{\ccr}{\ensuremath{\hat{\gamma}}}

\newcommand{\qAr}{\ensuremath{q_{\Ar}}}
\newcommand{\qBr}{\ensuremath{q_{\Br}}}
\newcommand{\qCr}{\ensuremath{q_{\Cr}}}

\newcommand{\xr}{\ensuremath{\hat{x}}}
\newcommand{\yr}{\ensuremath{\hat{y}}}

\newcommand{\xrd}{\ensuremath{\xr_{d}}}

\newcommand{\romset}{\ensuremath{\mathcal{R}}}

\newcommand{\hA}{\ensuremath{\hat{A}}}
\newcommand{\hB}{\ensuremath{\hat{B}}}
\newcommand{\hC}{\ensuremath{\hat{C}}}
\newcommand{\hE}{\ensuremath{\hat{E}}}
\newcommand{\hH}{\ensuremath{\hat{H}}}

\newcommand{\hX}{\ensuremath{\hat{X}}}

\newcommand{\obj}{\ensuremath{\mathcal{J}}}
\newcommand{\objmse}{\ensuremath{\obj_{\textnormal{MSE}}}}

\newcommand{\cH}{\ensuremath{\mathcal{H}}}
\newcommand{\cL}{\ensuremath{\mathcal{L}}}
\newcommand{\Htwo}{\ensuremath{\cH_{2}}}
\newcommand{\Ltwo}{\ensuremath{\cL_{2}}}
\newcommand{\Hinf}{\ensuremath{\cH_{\infty}}}
\newcommand{\Linf}{\ensuremath{\cL_{\infty}}}

\newcommand{\HtwoLtwo}{\ensuremath{\cH_{2} \otimes \cL_{2}}}

\newcommand{\pp}{\ensuremath{\mathsf{p}}}
\newcommand{\pset}{\ensuremath{\mathcal{P}}}
\newcommand{\npar}{\ensuremath{n_{\textnormal{p}}}}

\newcommand{\CC}{\ensuremath{\mathbb{C}}}
\newcommand{\RR}{\ensuremath{\mathbb{R}}}

\newcommand{\CCpar}{\ensuremath{\CC^{\npar}}}
\newcommand{\CCf}{\ensuremath{\CC^{\nfom}}}

\newcommand{\CCff}{\ensuremath{\CC^{\nfom \times \nfom}}}
\newcommand{\CCfi}{\ensuremath{\CC^{\nfom \times \nin}}}
\newcommand{\CCof}{\ensuremath{\CC^{\nout \times \nfom}}}
\newcommand{\CCoi}{\ensuremath{\CC^{\nout \times \nin}}}

\newcommand{\CCor}{\ensuremath{\CC^{\nout \times \nrom}}}
\newcommand{\CCri}{\ensuremath{\CC^{\nrom \times \nin}}}

\newcommand{\CCrr}{\ensuremath{\CC^{\nrom \times \nrom}}}

\newcommand{\RRpar}{\ensuremath{\RR^{\npar}}}
\newcommand{\RRf}{\ensuremath{\RR^{\nfom}}}
\newcommand{\RRi}{\ensuremath{\RR^{\nin}}}
\newcommand{\RRo}{\ensuremath{\RR^{\nout}}}

\newcommand{\RRff}{\ensuremath{\RR^{\nfom \times \nfom}}}
\newcommand{\RRfi}{\ensuremath{\RR^{\nfom \times \nin}}}
\newcommand{\RRof}{\ensuremath{\RR^{\nout \times \nfom}}}

\newcommand{\RRfr}{\ensuremath{\RR^{\nfom \times \nrom}}}
\newcommand{\RRor}{\ensuremath{\RR^{\nout \times \nrom}}}
\newcommand{\RRri}{\ensuremath{\RR^{\nrom \times \nin}}}

\newcommand{\RRrr}{\ensuremath{\RR^{\nrom \times \nrom}}}

\DeclarePairedDelimiter{\myparen}{\lparen}{\rparen}
\DeclarePairedDelimiter{\mybrack}{\lbrack}{\rbrack}
\DeclarePairedDelimiter{\mybrace}{\lbrace}{\rbrace}

\DeclarePairedDelimiter{\card}{\lvert}{\rvert}
\DeclarePairedDelimiter{\abs}{\lvert}{\rvert}
\DeclarePairedDelimiter{\norm}{\lVert}{\rVert}
\DeclarePairedDelimiterXPP{\mydiag}[1]{\operatorname{diag}}{\lparen}{\rparen}{}{#1}
\DeclarePairedDelimiterXPP{\normtwo}[1]{}{\lVert}{\rVert}{_{2}}{#1}
\DeclarePairedDelimiterXPP{\normF}[1]{}{\lVert}{\rVert}{_{\operatorname{F}}}{#1}
\DeclarePairedDelimiterXPP{\normHtwo}[1]{}{\lVert}{\rVert}{_{\Htwo}}{#1}
\DeclarePairedDelimiterXPP{\normLtwo}[1]{}{\lVert}{\rVert}{_{\Ltwo}}{#1}
\DeclarePairedDelimiterXPP{\normLtwomu}[1]{}{\lVert}{\rVert}{_{\Ltwo(\pset, \measure)}}{#1}
\DeclarePairedDelimiterXPP{\normLinf}[1]{}{\lVert}{\rVert}{_{\Linf}}{#1}
\DeclarePairedDelimiterXPP{\normLinfmu}[1]{}{\lVert}{\rVert}{_{\Linf(\pset, \measure)}}{#1}
\DeclarePairedDelimiterXPP{\ip}[2]{}{\langle}{\rangle}{}{#1, #2}
\DeclarePairedDelimiterXPP{\ipF}[2]{}{\langle}{\rangle}{_{\operatorname{F}}}{#1, #2}
\DeclarePairedDelimiterXPP{\ipHtwo}[2]{}{\langle}{\rangle}{_{\Htwo}}{#1, #2}
\DeclarePairedDelimiterXPP{\ipLtwo}[2]{}{\langle}{\rangle}{_{\Ltwo}}{#1, #2}
\DeclarePairedDelimiterXPP{\ipLtwomu}[2]{}{\langle}{\rangle}{_{\Ltwo(\pset, \measure)}}{#1, #2}
\DeclarePairedDelimiterXPP{\trace}[1]{\operatorname{tr}}{\lparen}{\rparen}{}{#1}
\DeclarePairedDelimiterXPP{\vecop}[1]{\operatorname{vec}}{\lparen}{\rparen}{}{#1}
\DeclarePairedDelimiterXPP{\Real}[1]{\operatorname{Re}}{\lparen}{\rparen}{}{#1}
\DeclarePairedDelimiterXPP{\myspan}[1]{\operatorname{span}}{\lbrace}{\rbrace}{}{#1}

\newcommand{\measure}{\ensuremath{\mu}}

\DeclareMathOperator{\dif}{d\!}
\DeclareMathOperator*{\MIN}{minimize}

\DeclareMathOperator*{\esssup}{ess\,sup}

\newcommand{\difm}[1]{\ensuremath{\dif{\measure(#1)}}}

\newcommand{\fundef}[3]{\ensuremath{#1 \colon #2 \to #3}}

\newcommand{\tran}{^{\operatorname{T}}}

\newcommand{\herm}{^{*}}
\newcommand{\mherm}{^{-*}}

\newcommand{\cBfblock}{\ensuremath{\mathbf{B}}}
\newcommand{\cCfblock}{\ensuremath{\mathbf{C}}}
\newcommand{\cBrblock}{\ensuremath{\hat{\cBfblock}}}
\newcommand{\cCrblock}{\ensuremath{\hat{\cCfblock}}}

\newcommand{\imag}{\boldsymbol{\imath}}

\newcommand{\dotvar}{\,\cdot\,}

\newcommand{\cX}{\ensuremath{\mathcal{X}}}

\let\le\leqslant%
\let\ge\geqslant%
\let\hat\widehat%
\definecolor{mplC0}{HTML}{1F77B4}
\definecolor{mplC1}{HTML}{FF7F0E}
\definecolor{mplC2}{HTML}{2CA02C}
\definecolor{mplC3}{HTML}{D62728}
\definecolor{mplC4}{HTML}{9467BD}
\definecolor{mplC5}{HTML}{8C564B}
\definecolor{mplC6}{HTML}{E377C2}
\definecolor{mplC7}{HTML}{7F7F7F}
\definecolor{mplC8}{HTML}{BCBD22}
\definecolor{mplC9}{HTML}{17BECF}

\pgfplotsset{compat=1.10}
\pgfplotscreateplotcyclelist{mpl}{
  mplC0, very thick \\
  mplC1, very thick, densely dashed \\
  mplC2, very thick, densely dotted \\
  mplC3, very thick, densely dashdotted \\
}
\pgfplotscreateplotcyclelist{mplmark}{
  mplC0, very thick, mark=o, every mark/.append style={solid} \\
  mplC1, very thick, densely dashed, mark=x, every mark/.append style={solid}, mark size=3 \\
  mplC2, very thick, densely dotted, mark=square, every mark/.append style={solid} \\
  mplC3, very thick, densely dashdotted, mark=diamond, every mark/.append style={solid} \\
}
\pgfplotscreateplotcyclelist{mplmarkonly}{
  draw=mplC0, fill=mplC0, only marks, mark=*, mark size=1 \\
  draw=mplC1, fill=mplC1, only marks, mark=square*, mark size=1 \\
  draw=mplC2, fill=mplC2, only marks, mark=triangle*, mark size=1.5 \\
}

\begin{document}

\title{\texorpdfstring{$\Ltwo$}{L2}-optimal Reduced-order Modeling Using
  Parameter-separable Forms\thanks{%
    This work was partially funded by the U.S. National Science Foundation under
    grant DMS-1923221.
    Parts of this material are based upon work supported by the National
    Science Foundation under Grant No.\ DMS-1929284 while the authors were in
    residence at the Institute for Computational and Experimental Research in
    Mathematics in Providence, RI, during the Spring 2020 Reunion Event for
    Model and Dimension Reduction in Uncertain and Dynamic Systems program.
  }}

\author{%
  \href{https://orcid.org/0000-0002-9437-7698}{%
    \includegraphics[scale=0.06]{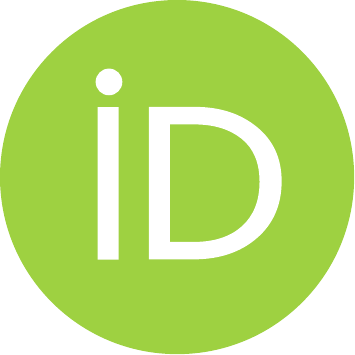}}\hspace{1mm}Petar~Mlinari\'c%
  \thanks{Department of Mathematics, Virginia Tech, Blacksburg, VA 24061
    (\texttt{mlinaric@vt.edu}).}
	\And
	\href{https://orcid.org/0000-0003-4564-5999}{%
    \includegraphics[scale=0.06]{orcid.pdf}}\hspace{1mm}Serkan~Gugercin%
  \thanks{Department of Mathematics and Division of Computational Modeling and
    Data Analytics, Academy of Data Science, Virginia Tech, Blacksburg, VA 24061
    (\texttt{gugercin@vt.edu}).}
}

\renewcommand{\headeright}{Technical Report}
\renewcommand{\undertitle}{Technical Report}
\renewcommand{\shorttitle}{$\Ltwo$-optimal Reduced-order Modeling}

\hypersetup{
  pdftitle={L2-optimal Reduced-order Modeling Using Parameter-separable Forms},
  pdfauthor={Petar~Mlinari\'c, Serkan~Gugercin},
}

\maketitle

\begin{abstract}
  We provide a unifying framework for $\Ltwo$-optimal reduced-order modeling for
  linear time-invariant dynamical systems and stationary parametric problems.
  Using parameter-separable forms of the reduced-model quantities,
  we derive the gradients of the $\Ltwo$ cost function with respect to the
  reduced matrices,
  which then allows a non-intrusive, data-driven, gradient-based descent
  algorithm to construct the optimal approximant using only output samples.
  By choosing an appropriate measure,
  the framework covers both continuous (Lebesgue) and discrete cost functions.
  We show the efficacy of the proposed algorithm via various numerical examples.
  Furthermore, we analyze under what conditions the data-driven approximant can
  be obtained via projection.
\end{abstract}


\section{Introduction}\label{sec:intro}
Consider a parameter-to-output mapping
\begin{equation}\label{eq:mapping}
  \fundef{\yf}{\pset}{\CCoi}, \quad
  \pp \mapsto \yf(\pp),
\end{equation}
where $\pset \subseteq \CCpar$ denotes the parameter space, and
$\npar, \nin, \nout$ are positive integers, representing the parameter, forcing
(input), and output dimensions of the underlying parametric model.
We are interested in cases where evaluating $\yf(\pp)$ for a given $\pp$ is
expensive
(thus causing a computational bottleneck in \emph{online} computations)
and we only have access to (the output) $\yf(\pp)$ without access to
an internal representation.

Our goal is to construct \iac{ddrom}
\begin{subequations}\label{eq:rom}
  \begin{align}
    \Ar(\pp) \xr(\pp) & = \Br(\pp), \\*
    \yr(\pp) & = \Cr(\pp) \xr(\pp),
  \end{align}
\end{subequations}
whose output $\yr(\pp)$ is significantly cheaper to evaluate compared to
$\yf(\pp)$ and
$\yr(\pp)$ is close to $\yf(\pp)$ for all $\pp \in \pset$.
In~\eqref{eq:rom} we have
$\Ar(\pp) \in \CCrr$,
$\Br(\pp) \in \CCri$,
$\Cr(\pp) \in \CCor$,
$\xr(\pp) \in \CCri$, and
$\yr(\pp) \in \CCoi$
where $\nrom$ is a modest integer so that evaluating $\yr(\pp)$
via~\eqref{eq:rom} is trivial.
The modeling structure in~\eqref{eq:rom} is motivated by model order reduction
for stationary parametric \acp{pde} and \ac{lti} dynamical systems as we briefly
explain next.

First, consider a parameterized linear \ac{pde} in the weak form
\begin{subequations}\label{eq:pde}
  \begin{align}
    a(\xi(\pp), \zeta; \pp)
    & =
      f(\zeta; \pp), \quad
      \forall \zeta \in \cX, \\*
    q(\pp)
    & =
      l(\xi(\pp); \pp),
  \end{align}
\end{subequations}
where
$\pp \in \pset \subseteq \RRpar$ is the parameter,
$\cX$ is a real Hilbert space,
$\xi(\pp) \in \cX$ is the solution, and
$q(\pp) \in \RR$ is the quantity of interest.
Furthermore,
$\fundef{a(\dotvar, \dotvar; \pp)}{\cX \times \cX}{\RR}$ is a coercive and
continuous bilinear form and
$\fundef{f(\dotvar; \pp), l(\dotvar; \pp)}{\cX}{\RR}$ are bounded linear
functionals
for all $\pp \in \pset$.
In the simple case with
$\npar = 1$,
$a(\xi, \zeta; \pp) = a_1(\xi, \zeta) + \pp a_2(\xi, \zeta)$,
$f(\zeta; \pp) = f(\zeta)$, and
$l(\xi; \pp) = l(\xi)$,
after a Galerkin projection onto an $\nfom$-dimensional subspace
$\myspan{\xi_1, \xi_2, \ldots, \xi_{\nfom}} \subset \cX$
(e.g., constructed by a finite element discretization),
we obtain a finite-dimensional model
\begin{subequations}\label{eq:stationary}
  \begin{align}
    (\cAf_1 + \pp \cAf_2) \xf(\pp) & = \cBf, \label{eq:stationary-state} \\*
    \yf(\pp) & = \cCf \xf(\pp),
  \end{align}
\end{subequations}
where $\xf(\pp) \in \RRf$ is the projected solution,
$\yf(\pp) \in \RR$ is the output approximating $q(\pp)$, and
$\cAf_1, \cAf_2 \in \RRff$ and
$B, C\tran \in \RR^{\nfom \times 1}$ are given component-wise by
${[\cAf_1]}_{ij} = a_1(\xi_j, \xi_i)$,
${[\cAf_2]}_{ij} = a_2(\xi_j, \xi_i)$,
${[\cBf]}_{i1} = f(\xi_i)$, and
${[\cCf]}_{1j} = l(\xi_j)$.
If, for example, $\cAf_1 + \pp \cAf_2$ is invertible for every $\pp \in \pset$,
then the parameter-to-output mapping in this case is given by
$\yf(\pp) = \cCf {(\cAf_1 + \pp \cAf_2)}^{-1} \cBf$.
This problem corresponds to $\nin = \nout = 1$.
Note that $\nin$ here represents the number of right-hand sides
in~\eqref{eq:stationary-state}, i.e.,
the number of forcing terms.
If, in addition, the problem~\eqref{eq:pde} is compliant, i.e.,
$a(\dotvar, \dotvar; \pp)$ is symmetric and $l = f$,
then $\cAf_1$ and $\cAf_2$ are symmetric and $C = B\tran$.

Now, consider \iac{lti} dynamical system described in state space as
\begin{subequations}\label{eq:lti}
  \begin{align}
    E \dot{\xf}(t) & = A \xf(t) + B u(t), \quad x(0) = 0, \\*
    \yf(t) & = C \xf(t),
  \end{align}
\end{subequations}
where
$t \in \RR$ is the time,
$u(t) \in \RRi$ is the input,
$\xf(t) \in \RRf$ is the state,
$\yf(t) \in \RRo$ is the output,
$E, A \in \RRff$,
$B \in \RRfi$, and
$C \in \RRof$.
By applying the Laplace transform to~\eqref{eq:lti},
we obtain $Y(s) = H(s) U(s)$,
where $U$ and $Y$ are, respectively, the Laplace transforms of $u$ and $\yf$.
Furthermore, $H(s) \in \CCoi$ is given by
\begin{equation}\label{eq:tf}
  H(s) = C {(s E - A)}^{-1} B
\end{equation}
and is called the transfer function, which is
at the heart of systems-theoretic approaches to optimal approximation of
\ac{lti} systems~\cite{AntBG20,Ant05}.%
\begin{subequations}\label{eq:Hsys}
We can rewrite $H(s)$ as
  \begin{align}
    (s E - A) X(s) & = B, \\*
    H(s) & = C X(s),
  \end{align}
\end{subequations}
for any $s \in \CC$ such that $s E - A$ is invertible and
$X(s) \in \CCfi$.

Therefore, both mappings, namely
$\pp \mapsto q(\pp)$ in~\eqref{eq:pde} and
$s \mapsto H(s)$ in~\eqref{eq:tf},
are examples of parameter-to-output mappings~\eqref{eq:mapping}
we consider in this paper.
Both models~\eqref{eq:stationary}
(resulting from discretization of a stationary parametric \ac{pde})
and~\eqref{eq:Hsys}
(frequency domain formulation of \iac{lti} system)
can be examined using the form
\begin{subequations}\label{eq:fom}
  \begin{align}
    \Af(\pp) \xf(\pp) & = \Bf(\pp), \\*
    \yf(\pp) & = \Cf(\pp) \xf(\pp),
  \end{align}
\end{subequations}
where $\pp \in \pset \subseteq \CCpar$ is the parameter,
$\xf(\pp) \in \CCfi$ is the state,
$\yf(\pp) \in \CCoi$ is the output,
$\Af(\pp) \in \CCff$,
$\Bf(\pp) \in \CCfi$, and
$\Cf(\pp) \in \CCof$.
Many applications require to solve the model~\eqref{eq:fom} in real time or for
many parameter values,
which incurs a computational bottleneck due to the large-scale dimension of the
underlying state-space.
The goal of model order reduction for parametric \acp{pde} and for \ac{lti}
systems is to replace~\eqref{eq:fom} with \iac{rom},
which motivates us to approximate the mapping~\eqref{eq:mapping} by the
\ac{ddrom} of the form~\eqref{eq:rom}.

Thus the framework we consider handles a wide range of problems
(stationary or dynamic) including those of the form in~\eqref{eq:fom}.
We revisit both problems~\eqref{eq:stationary} and~\eqref{eq:Hsys} throughout
the paper and illustrate how the theory applies in either case.
Furthermore, even though the motivation comes from \acp{fom} of the from
in~\eqref{eq:fom},
the approximation framework we develop below only requires access to the
parameter-to-output mapping~\eqref{eq:mapping} and
\emph{not} to the full-order operators $\Af, \Bf, \Cf$ and state $\xf$.
Thus, we work with a non-intrusive parameter/output data-driven formulation.
Therefore, we refer to our methodology as ``reduced-order modeling'' instead of
``model order reduction''.

It is worth mentioning that a similar setting of a parameter-to-output mapping
appears in, e.g., active subspaces~\cite{ConDW14}.
As~\cite{ConDW14} focuses on parameter reduction,
we believe it could be used in combination with the approach we propose here to
develop a combined parameter and state reduction method
(such as in~\cite{HimO14}).

There are different ways of measuring the distance between $\yf$ and $\yr$.
For instance, \ac{rb} methods~\cite{BenGQetal20b} are based on the $\Linf$ norm
\begin{equation*}
  \normLinf{\yf - \yr}
  = \sup_{\pp \in \pset} \, \normF*{\yf(\pp) - \yr(\pp)},
\end{equation*}
where $\normF{\dotvar}$ is the Frobenius norm.
For \ac{lti} systems, the corresponding measure is the $\Hinf$ norm and we refer
the reader to recent optimization-based algorithms for (structure-preserving)
$\Hinf$-optimal model order reduction~\cite{Mit16,SchV20}.

In this paper, we focus on a different norm.
Motivated by the work on $\Htwo$-optimal model order
reduction~\cite{GugAB06,GugAB08,AntBG10} for non-parametric \ac{lti} systems,
and extensions to $\HtwoLtwo$-optimal model order
reduction~\cite{BauBBetal11,Pet13,Gri18,HunMMS21} for parametric \ac{lti}
systems,
we are interested in $\Ltwo$-optimal reduced-order modeling for parametric
problems~\eqref{eq:mapping}.
Specifically,
we are interested in finding \iac{ddrom}~\eqref{eq:rom} that minimizes the
$\Ltwo$ error
\begin{equation*}
  \normLtwo{\yf - \yr}
  = \myparen*{\int_{\pset} \normF*{\yf(\pp) - \yr(\pp)}^2 \dif{\pp}}^{1/2}.
\end{equation*}
The goal is to develop the analysis
(and the resulting computational tools)
so that the framework equally applies
to parametric stationary problems as in~\eqref{eq:stationary} and
to dynamical system as in~\eqref{eq:lti}
by the proper definition of the parameter space and error measure.
Additionally,
we want the analysis to be applicable to more general measures $\measure$ over
the parameter space $\pset$, i.e.,
minimizing
\begin{equation}\label{eq:l2error}
  \normLtwomu{\yf - \yr}
  = \myparen*{\int_{\pset} \normF*{\yf(\pp) - \yr(\pp)}^2 \difm{\pp}}^{1/2}.
\end{equation}
For example, $\measure$ could be a probability measure over $\pset$ and the
parameter $\pp$ could be treated as a random variable.
Another example of a measure is a discrete measure
$\measure_{\textnormal{d}} = \sum_{i = 1}^N \delta_{\pp_i}$
where $\delta_{x}$ is the Dirac measure ($\delta_x(A) = \card{\{x\} \cap A}$)
and $\pp_1, \pp_2, \ldots, \pp_N$ are some parameter values,
which results in the error measure
\begin{equation*}
  \norm{\yf - \yr}_{\Ltwo(\pset, \measure_{\textnormal{d}})}
  = \myparen*{\sum_{i = 1}^N \normF*{\yf(\pp_i) - \yr(\pp_i)}^2}^{1/2}.
\end{equation*}
Therefore, we develop theoretical results that hold for both continuous and
discrete objective functions.

We note that the objective~\eqref{eq:l2error} is reminiscent of operator
inference~\cite{PehW16}.
However, the fundamental difference is that operator inference solves a linear
least-squares problem,
while we solve a nonlinear optimization problem.
This is due to the fact that while operator inference minimizes the residual,
our goal is to minimize the output error.
Furthermore, operator inference requires, in its original formulation,
full state snapshots,
while our approach only needs output measurements.
Additionally, operator inference is usually posed in the time domain
unlike our $\Ltwo$ measure which would be posed in the Laplace/frequency domain.

The main contributions of the paper are as follows:
\begin{itemize}
\item We develop a unifying formulation for $\Ltwo$-optimal data-driven
  reduced-order modeling,
  which applies to a wide range of problems with an appropriate definition of
  the measure space.
\item We derive explicit formulae for gradients of the $\Ltwo$ approximation
  error with respect to the matrices of the \ac{ddrom}.
  These gradient computations require access only to the model output without
  internal (state) samples.
\item Based on the gradient formulae, we develop a data-driven, gradient-based
  algorithm for $\Ltwo$-optimal reduced-order modeling.
\item We extend the framework to a discrete least-squares error function.
\item We analyze and give conditions under which the $\Ltwo$-optimal \ac{ddrom}
  can be obtained via projection.
\end{itemize}

The rest of the paper is organized as follows:
In \Cref{sec:preliminaries} we briefly recall projection-based model order
reduction, the most common framework for intrusive model order reduction.
We state the structured $\Ltwo$-optimal reduced-order modeling problem in
\Cref{sec:optimal-mor},
and derive the gradients of the squared $\Ltwo$ error with respect to the
matrices of the \ac{ddrom}.
Furthermore, there we discuss a generic optimization-based algorithm for
$\Ltwo$-optimal reduced-order modeling.
In \Cref{sec:numerics-cont} we focus on the continuous objective function and
provide numerical examples.
Then, we discuss discrete objective function in
\Cref{sec:numerics-disc},
where we demonstrate the generic algorithm on further examples.
In \Cref{sec:mor-by-projection} we return to projection-based model order
reduction and discuss whether $\Ltwo$-optimal \acp{ddrom} are projection-based.
Finally, \Cref{sec:conclusion} gives concluding remarks.

\subsection{Projection-based Model Order Reduction}%
\label{sec:preliminaries}
Even though our framework is data-driven and does not start with or need a
\ac{fom} to reduce,
in this section, we briefly recall the basics of the projection-based model
order reduction methods to help motivate the structure enforced on the
\ac{ddrom}~\eqref{eq:rom}.

For a \ac{fom}~\eqref{eq:fom},
the Petrov-Galerkin projection framework is one of the most common ways to
construct the \ac{rom}~\eqref{eq:rom}.
In this framework, given the \ac{fom}~\eqref{eq:fom},
one chooses two $\nrom$-dimensional subspaces of $\RRf$,
spanned by the columns of $V, W \in \RRfr$, and
constructs the \ac{rom}~\eqref{eq:rom} by
\begin{equation}\label{eq:proj}
  \Ar(\pp) = W\tran \Af(\pp) V, \quad
  \Br(\pp) = W\tran \Bf(\pp), \quad
  \Cr(\pp) = \Cf(\pp) V.
\end{equation}
If $V$ and $W$ span the same subspace,
this is called a Galerkin projection.

Even though $\Ar(\pp) \in \CCrr$, $\Br(\pp) \in \CCri$, $\Cr(\pp) \in \CCor$
in~\eqref{eq:proj} have the reduced row and/or column dimensions,
evaluating them for a new parameter value $\pp$ requires operations in the full
dimension $\nfom$.
Thus, for efficient computation of the \ac{rom},
it is often assumed that the \ac{fom} matrices have a parameter-separable form
(or that it can be approximated by one, e.g., using the empirical interpolation
method~\cite{BarMNP04}), i.e.,
\begin{equation}\label{eq:fom-par-sep-form}
  \Af(\pp) = \sum_{i = 1}^{\qAf} \caf_i(\pp) \cAf_i, \quad
  \Bf(\pp) = \sum_{j = 1}^{\qBf} \cbf_j(\pp) \cBf_j, \quad
  \Cf(\pp) = \sum_{k = 1}^{\qCf} \ccf_k(\pp) \cCf_k,
\end{equation}
where
$\qAf, \qBf, \qCf$ are small positive integers,
$\fundef{\caf_i, \cbf_j, \ccf_k}{\pset}{\CC}$ are given functions that are easy
to evaluate, and
$\cAf_i \in \RRff$,
$\cBf_j \in \RRfi$,
$\cCf_k \in \RRof$
are constant matrices.
Then, one computes the following \ac{rom} matrices only once
\begin{equation}\label{eq:galerkin}
  \cAr_i = W\tran \cAf_i V, \quad
  \cBr_j = W\tran \cBf_j, \quad
  \cCr_k = \cCf_k V,
\end{equation}
and the \ac{rom}~\eqref{eq:rom} is constructed efficiently as
\begin{equation}\label{eq:pg-proj-rom}
  \Ar(\pp) = \sum_{i = 1}^{\qAf} \caf_i(\pp) \cAr_i, \quad
  \Br(\pp) = \sum_{j = 1}^{\qBf} \cbf_j(\pp) \cBr_j, \quad
  \Cr(\pp) = \sum_{k = 1}^{\qCf} \ccf_k(\pp) \cCr_k.
\end{equation}
Thus, the full-order operators $\Af(\pp), \Bf(\pp), \Cf(\pp)$ are avoided when
solving the \ac{rom}.
There are many projection-based model order reduction methods and thus many
different ways of computing $V$ and $W$; see,
e.g.,~\cite{BenGQetal20a,BenGQetal20b,Ant05,AntBG20,BenGW15,QuaMN16,HesRS16,Benetal17}.
We revisit some of these methods in more detail in \Cref{sec:numerics-cont}.

\section{\texorpdfstring{$\Ltwo$}{L2}-optimal Reduced-order Modeling}%
\label{sec:optimal-mor}
In this section, we first establish the setting of the optimal reduced-order
modeling problem we consider and
prove the main theoretical result that forms the foundation of the proposed
algorithm.

\subsection{Setting}
We are interested in approximating a parameter-to-output
mapping~\eqref{eq:mapping} by \iac{ddrom}~\eqref{eq:rom}.
Although the motivation comes from the form of \acp{fom} as in~\eqref{eq:fom},
the framework we develop here does not require the full-order operators
$\Af$, $\Bf$, and $\Cf$ or the full-order state $\xf$.
Instead we only need (the samples of) the output $\yf$.
In other words,
we develop an optimal data-driven approximation formulation
that only uses the parameter/output samples of the model under consideration.
More specifically,
we consider the \acp{fom} accessible only via
a complex-valued output function $\fundef{\yf}{\pset}{\CCoi}$
where $\pset \subseteq \CCpar$ and
$(\pset, \Sigma, \measure)$ is a measure space.

We make some technical assumptions on the \ac{fom} valid for the general setup
we consider here.
Then by revisiting the specific \acp{fom} in~\eqref{eq:stationary}
and~\eqref{eq:Hsys},
we show that these are common assumptions and
automatically hold in most cases.
\begin{assumption}[\Ac{fom} assumptions]\label{assumption:fom}
  Let $(\pset, \Sigma, \measure)$ be a measure space and
  $\fundef{\yf}{\pset}{\CCoi}$ a measurable function.
  \begin{itemize}
  \item The set $\pset \subseteq \CCpar$ is closed under conjugation
    ($\overline{\pp} \in \pset$ for all $\pp \in \pset$),
  \item The $\sigma$-algebra $\Sigma$ is closed under conjugation
    ($\overline{S} \in \Sigma$ for all $S \in \Sigma$),
  \item The measure $\measure$ is closed under conjugation
    ($\measure(\overline{S}) = \measure(S)$ for all $S \in \Sigma$),
  \item The function $\yf$ is square-integrable
    ($\normLtwomu{\yf} < \infty$) and
    closed under conjugation
    ($\overline{\yf(\pp)} = \yf(\overline{\pp})$ for all $\pp \in \pset$).
  \end{itemize}
\end{assumption}
\begin{example}\label{ex:fom}
  We revisit the two basic examples from \Cref{sec:intro} under the setting of
  \Cref{assumption:fom}.
  First consider the \ac{fom}~\eqref{eq:stationary} resulting from the
  discretization of a stationary parametric \ac{pde}.
  Let $\pset = [a, b] \subset \RR$ and $\measure$ be the Lebesgue measure.
  Then, \Cref{assumption:fom} holds
  if $\cAf_1 + \pp \cAf_2$ is invertible for all $\pp \in \pset$,
  a common assumption.

  Now, recall the transfer function~\eqref{eq:tf} of \iac{lti}
  system~\eqref{eq:lti} formulated as a parametric stationary
  problem~\eqref{eq:Hsys}.
  One commonly used system norm is the Hardy $\Htwo$ norm $\normHtwo{\dotvar}$,
  which gives the output bound $\normLinf{\yf} \le \normHtwo{H} \normLtwo{u}$.
  The norm can be formulated as
  \(
    \normHtwo{H}
    =
    \myparen{
      \frac{1}{2 \pi} \int_{-\infty}^{\infty}
      \normF{H(\imag \omega)}^2 \dif{\omega}
    }^{1/2},
  \)
  assuming that $E$ is invertible and
  all the eigenvalues of $E^{-1} A$ have negative real parts,
  where $\imag$ denotes the imaginary unit.
  Therefore, to have $\normHtwo{H} = \normLtwomu{H}$,
  we can take
  $\pp = s$,
  $\pset = \imag \RR$, and
  $\measure = \frac{1}{2 \pi} \lambda_{\imag \RR}$,
  where $\lambda_{\imag \RR}$ is the Lebesgue measure over $\imag \RR$.
  A sufficient condition for \Cref{assumption:fom} to hold is that
  $E$ be invertible and $E^{-1} A$ have no eigenvalues on the imaginary axis,
  which are weaker assumptions than those needed to define the $\Htwo$ norm.
  These are also common assumptions in the systems-theoretic setting.
  One can indeed allow $E$ to be singular
  (i.e., allow systems of differential algebraic equations)
  as has been done in many earlier works~\cite{MehS05,GugSW13}.
  However, to keep the notation and discussion concise,
  we assume $E$ to be invertible.
\end{example}

\subsection{Optimization Problem with Parameter-separable Forms}
Given the parameter-to-output mapping in~\eqref{eq:mapping},
our goal is to find \iac{ddrom}~\eqref{eq:rom}
that minimizes the output $\Ltwo$ error~\eqref{eq:l2error}.
As discussed in \Cref{sec:preliminaries}, many \acp{fom} have a
parameter-separable form as in~\eqref{eq:fom-par-sep-form} and
this form is preserved in the classical Petrov-Galerkin projection-based
\acp{rom}.
Inspired by this formulation, in our $\Ltwo$-optimal \ac{ddrom} setting,
we search for a structured \ac{ddrom} with parameter-separable form
\begin{equation}\label{eq:rom-param-sep-form}
  \Ar(\pp) = \sum_{i = 1}^{\qAr} \car_i(\pp) \cAr_i, \quad
  \Br(\pp) = \sum_{j = 1}^{\qBr} \cbr_j(\pp) \cBr_j, \quad
  \Cr(\pp) = \sum_{k = 1}^{\qCr} \ccr_k(\pp) \cCr_k,
\end{equation}
where
$\qAr, \qBr, \qCr$ are positive integers,
$\fundef{\car_i, \cbr_j, \ccr_k}{\pset}{\CC}$ are given measurable functions,
and
$\cAr_i \in \RRrr$,
$\cBr_j \in \RRri$,
$\cCr_k \in \RRor$
are the (\ac{ddrom}) matrices we want to compute to minimize the $\Ltwo$
error~\eqref{eq:l2error}.
Note that even though the \ac{ddrom} structure is inspired by the
parameter-separable structures appearing in many \acp{fom} and preserved in
projection-based \acp{rom},
here we make no assumptions on the form of the \ac{fom},
but only on the form of the \ac{ddrom}.
The subtle notational difference between~\eqref{eq:pg-proj-rom}
and~\eqref{eq:rom-param-sep-form},
namely the ``hatted'' scalar functions,
aims to highlight that unlike in the projection-based
\ac{rom}~\eqref{eq:pg-proj-rom},
where the scalar functions match those of the \ac{fom},
in the \ac{ddrom}~\eqref{eq:rom-param-sep-form},
we have freedom in choosing them.
We also note that this parameter-separable structure appears
in~\cite{BenGP19,HunMMS21} as well.

We make the following assumptions on the scalar functions
$\car_i, \cbr_j, \ccr_k$ appearing in the
\ac{ddrom}~\eqref{eq:rom-param-sep-form}.
As we did for \Cref{assumption:fom},
later we discuss that these assumptions indeed hold trivially in many cases.
\begin{assumption}[Scalar functions]\label{assumption:scalar-funcs}
  Let $\pset$ and $\measure$ satisfy \Cref{assumption:fom}.
  The functions $\fundef{\car_i, \cbr_j, \ccr_k}{\pset}{\CC}$ are
  measurable,
  closed under conjugation, and satisfy
  \begin{equation}\label{eq:abc-l2-bounded}
    \int_{\pset}
    \myparen*{
      \frac{
        \sum_{j = 1}^{\qBr} \abs*{\cbr_j(\pp)}
        \sum_{k = 1}^{\qCr} \abs*{\ccr_k(\pp)}
      }{
        \sum_{i = 1}^{\qAr} \abs{\car_i(\pp)}
      }
    }^{\!\!2}
    \difm{\pp}
    < \infty.
  \end{equation}
\end{assumption}
Now based on \cref{assumption:fom} and \cref{assumption:scalar-funcs},
we introduce the set of allowable \ac{ddrom} matrices.
\begin{definition}\label{def:rom-set}
  Let $\pset$ and $\measure$ satisfy \Cref{assumption:fom} and
  $\fundef{\car_i, \cbr_j, \ccr_k}{\pset}{\CC}$ satisfy
  \Cref{assumption:scalar-funcs}.
  Next, let
  $R = {(\RRrr)}^{\qAr} \times {(\RRri)}^{\qBr} \times {(\RRor)}^{\qCr}$
  be the set of all tuples of \ac{ddrom} matrices
  $(\cAr_1, \ldots, \cAr_{\qAr},
  \cBr_1, \ldots, \cBr_{\qBr},
  \cCr_1, \ldots, \cCr_{\qCr})
  \eqqcolon (\cAr_i, \cBr_j, \cCr_k)$.
  Then, we define the set $\romset$ of allowable \ac{ddrom} matrices as
  \begin{equation}\label{eq:setR}
    \romset =
    \mybrace*{
      \myparen*{\cAr_i, \cBr_j, \cCr_k} \in R :
      \esssup_{\pp \in \pset} \, \normF*{\car_i(\pp) \Ar(\pp)^{-1}}
      < \infty,\
      i = 1, 2, \ldots, \qAr
    },
  \end{equation}
  where $\Ar$ is as in~\eqref{eq:rom-param-sep-form}.
\end{definition}
\begin{example}
  We want to illustrate that the conditions in~\eqref{eq:abc-l2-bounded}
  and~\eqref{eq:setR} do indeed hold trivially in many cases and
  thus are not restrictive.
  Continuing with \Cref{ex:fom},
  we want to cover the analogous \acp{ddrom}.
  Starting with the parametric stationary problem, we consider
  \begin{align*}
    \myparen*{\cAr_1 + \pp \cAr_2} \xr(\pp) & = \cBr, \\*
    \yr(\pp) & = \cCr \xr(\pp).
  \end{align*}
  For this case, we have
  \begin{equation*}
    \qAr = 2,\ \car_1(\pp) = 1,\ \car_2(\pp) = \pp; \quad
    \qBr = 1,\ \cbr_1(\pp) = 1; \textnormal{ and } \quad
    \qCr = 1,\ \ccr_1(\pp) = 1.
  \end{equation*}
  Therefore, the condition~\eqref{eq:abc-l2-bounded} becomes
  $\int_a^b \frac{1}{{(1 + \abs{\pp})}^2} \dif{\pp} < \infty$,
  which holds true.
  The condition in~\eqref{eq:setR} states that
  $\pp \mapsto \Ar(\pp)^{-1}$ and $\pp \mapsto \pp \Ar(\pp)^{-1}$ are
  bounded over $[a, b]$.
  The necessary and sufficient condition is that $\Ar(\pp)$ be invertible for
  all $\pp \in \pset$.

  Next, we consider a reduced-order \ac{lti} system
  \begin{subequations}\label{eq:romH}
    \begin{align}
      \myparen*{s \hE - \hA} \hX(s) & = \hB, \\*
      \hat{H}(s) & = \hC \hX(s).
    \end{align}
  \end{subequations}
  For this case, we have
  \begin{equation*}
    \qAr = 2,\ \car_1(\pp) = \pp,\ \car_2(\pp) = -1; \quad
    \qBr = 1,\ \cbr_1(\pp) = 1; \textnormal{ and } \quad
    \qCr = 1,\ \ccr_1(\pp) = 1,
  \end{equation*}
  with $\pp = s = \imag \omega$.
  The condition~\eqref{eq:abc-l2-bounded} becomes
  $\int_{-\infty}^{\infty}
  \frac{1}{{(\abs{\omega} + 1)}^2} \dif{\omega}
  < \infty$,
  which also holds.
  Finally, the condition in~\eqref{eq:setR} states that
  $s \mapsto (s \hE - \hA)^{-1}$ and
  $s \mapsto s (s \hE - \hA)^{-1} = (\hE - \frac{1}{s} \hA)^{-1}$
  are bounded over $\imag \RR$.
  This is equivalent to the invertibility of $\hE$ and
  $\hE^{-1} \hA$ having no eigenvalues in $\imag \RR$,
  similar to the \ac{fom} discussed in \Cref{ex:fom}.
\end{example}
The choice of functions $\fundef{\car_i, \cbr_j, \ccr_k}{\pset}{\CC}$ is
flexible as long as they satisfy \Cref{assumption:scalar-funcs}.
In many cases, such as for \ac{lti} systems as above,
physically-inspired choices are immediately available from the underlying
physics.

For the analysis in \Cref{sec:grad},
it is important to establish that the set $\romset$~\eqref{eq:setR} is open and
that it forms a set of feasible \acp{ddrom}.
This is what we do next.
\begin{lemma}\label{lem:rom-set-open-yr-bounded}
  The set $\romset$~\eqref{eq:setR} in \Cref{def:rom-set} is open.
  Moreover, for all $\yr$ defined by \iac{ddrom}
  $(\cAr_i, \cBr_j, \cCr_k) \in \romset$,
  we have that $\yr$ is square-integrable.
\end{lemma}
\begin{proof}
  Proof of this result is given in \Cref{sec:proofs}.
\end{proof}

\subsection{Computing the Gradients}%
\label{sec:grad}
In this section, we derive one of the main results of this paper,
mainly the gradients of the $\Ltwo$ cost function~\eqref{eq:l2error} with
respect to the \ac{ddrom} matrices $\cAr_i, \cBr_j, \cCr_k$.
These gradient formulae form the foundation of the $\Ltwo$-optimal reduced-order
modeling algorithm we develop.

Given the \ac{fom} as a parameter-to-output mapping $\yf$~\eqref{eq:mapping},
we want to construct \iac{ddrom}~\eqref{eq:rom} with the structured
reduced-order matrices of the form~\eqref{eq:rom-param-sep-form}.
Furthermore, we look for \iac{ddrom} belonging to $\romset$ from
\Cref{def:rom-set},
since this guarantees that the squared $\Ltwo$ error is well-defined and
differentiable over $\romset$ (as shown in \cref{thm:gradients})
without substantial restrictions on the form of the \ac{ddrom}.
Thus, we consider the structured $\Ltwo$-optimization problem
\begin{align*}
  \MIN_{(\cAr_i, \cBr_j, \cCr_k) \in \romset} \quad
  & \obj\myparen*{\cAr_i, \cBr_j, \cCr_k}
    = \normLtwomu{\yf - \yr}^2.
\end{align*}
In our analysis below,
we also employ the reduced-order \emph{dual state} $\xrd(\pp)$,
satisfying the reduced-order dual state equation
$\Ar(\pp)\herm \xrd(\pp) = \Cr(\pp)\herm$~\cite{FenB19},
where $(\dotvar)\herm$ denotes the conjugate transpose of a matrix.

Before establishing the gradients of the objective function $\obj$ with respect
to the \ac{ddrom} matrices, we recall some notation.
For a Fr\'echet differentiable function $\fundef{f}{U}{\RR}$,
defined on an open subset $U$ of a Hilbert space $H$ with inner product
$\ip{\dotvar}{\dotvar}$,
the gradient of $f$ at $x$,
denoted $\nabla f(x)$,
is the unique element of $H$ satisfying
$f(x + h) = f(x) + \ip{\nabla f(x)}{h} + o(\norm{h})$,
where $g(h) = o(\norm{h})$ denotes that $\lim_{h \to 0} g(h)/\norm{h} = 0$.
For a multivariate function $f(x_1, x_2, \ldots, x_k)$,
partial gradients $\nabla_{x_i} f(x_1, x_2, \ldots, x_k)$ are defined in a
similar way.
\begin{theorem}\label{thm:gradients}
  Let $\pset$, $\measure$, and $\yf$ satisfy \Cref{assumption:fom} and
  $(\cAr_i, \cBr_j, \cCr_k)$ be a tuple of \ac{ddrom} matrices belonging to
  $\romset$~\eqref{eq:setR}.
  Then, the gradients of $\obj$ with respect to the \ac{ddrom} matrices are
  \begin{align*}
    \nabla_{\cAr_i} \obj = {}
    &
      2 \int_{\pset} \car_i(\overline{\pp})
      \xrd(\pp)
      \mybrack*{\yf(\pp) - \yr(\pp)}
      \xr(\pp)\herm
      \difm{\pp},
    & i = 1, 2, \ldots, \qAr, \\*
    \nabla_{\cBr_j} \obj = {}
    &
      2 \int_{\pset} \cbr_j(\overline{\pp})
      \xrd(\pp)
      \mybrack*{\yr(\pp) - \yf(\pp)}
      \difm{\pp},
    & j = 1, 2, \ldots, \qBr, \\*
    \nabla_{\cCr_k} \obj = {}
    &
      2 \int_{\pset} \ccr_k(\overline{\pp})
      \mybrack*{\yr(\pp) - \yf(\pp)}
      \xr(\pp)\herm
      \difm{\pp},
    & k = 1, 2, \ldots, \qCr.
  \end{align*}
\end{theorem}
\begin{proof}
  The expressions follow from the use of the definition of the gradient and
  using the assumptions in \Cref{assumption:fom},
  \Cref{assumption:scalar-funcs}, \Cref{def:rom-set} and
  the result of \Cref{lem:rom-set-open-yr-bounded} to show differentiability.
  The full proof is given in \Cref{sec:proofstheorem}.
\end{proof}

Note that these gradient computations do not require access to the full-order
matrices or the full-order state.
They are computed directly from the evaluations of the output $y(\pp)$ of the
\ac{fom}.
This allows us to develop
a non-intrusive, data-driven, optimization-based, reduced-order modeling
algorithm that only needs access to the output $\yf(\pp)$ of the \ac{fom}.
With data-driven access to these gradient evaluations,
we can design a variety of optimization algorithms to construct an
$\Ltwo$-optimal \ac{ddrom} for different scenarios.
We discuss these details in the next subsection.
\begin{remark}\label{rem:future-work}
  The $\Ltwo$ norm in~\eqref{eq:l2error} recovers both the $\Htwo$
  norm for non-parametric \ac{lti} systems (see \Cref{ex:fom}) and
  the $\HtwoLtwo$ norm for parametric \ac{lti} systems
  by the appropriate choice of the parameter space $\pset$ and measure
  $\measure$.
  In particular, \Cref{thm:gradients} has implications for interpolatory optimal
  reduced-order modeling of dynamical systems~\cite{AntBG20,GugAB08,HunMMS21}
  and unifies optimal interpolation conditions for $\Htwo$-optimal and
  $\HtwoLtwo$-optimal model order reduction of (parametric) \ac{lti} systems.
  These details together with interpolatory optimality
  conditions for approximating parametric stationary problems as
  in~\eqref{eq:stationary} can be found in~\cite{MliG22b}.
\end{remark}

\subsection{Algorithmic Details}
In this section, we describe our proposed algorithm for $\Ltwo$-optimal
reduced-order modeling using parameter-separable form ($\Ltwo$-Opt-PSF).
The pseudocode is given in \Cref{alg:l2opt}.
As discussed earlier, from the \ac{fom},
we only need output evaluations or samples, i.e., only $y(\pp)$ is needed.
$\Ltwo$-Opt-PSF does not require access to the internal state variables or
full-order operators.

\begin{algorithm}
  \caption{$\Ltwo$-Opt-PSF}%
  \label{alg:l2opt}
  \begin{algorithmic}[1]
    \Require%
      Parameter-to-output mapping $\yf$,
      initial guess for \iac{ddrom} $(\cAr_i, \cBr_j, \cCr_k)$,
      maximum number of iterations $\mathtt{maxit}$,
      tolerance $\mathtt{tol} > 0$.
    \Ensure%
    \Ac{ddrom} $(\cAr_i, \cBr_j, \cCr_k)$.
    \State%
      Set $\yr^{(0)}$ as the output of the initial \ac{ddrom}.
    \For{$i$ in $1, 2, \ldots, \mathtt{maxit}$}
      \State%
        \parbox[t]{\linewidth-\algorithmicindent}{%
          Compute a new \ac{ddrom} with output $\yr^{(i)}$
          using a step of a gradient-based optimization method,
          with the squared $\Ltwo$ error~\eqref{eq:l2error} as the objective
          function and gradients based on \Cref{thm:gradients}.%
        }\label{line:l2opt-grad}
      \If{$\normLtwomu{\yr^{(i - 1)} - \yr^{(i)}} / \normLtwomu{\yr^{(i)}}
        \le \mathtt{tol}$}%
        \label{line:l2opt-stop}
      \State%
        Exit the \textbf{for} loop.
      \EndIf%
    \EndFor%
    \State%
      Return the last computed \ac{ddrom}.
  \end{algorithmic}
\end{algorithm}

Some comments on the pseudocode are in order.
First, the choice of the initial guess has an impact on the final result,
as with any other non-convex optimization problem.
Second, we do not specify the gradient-based optimization method used in
Step~\ref{line:l2opt-grad}.
Third, the computations of the objective function and its gradient are not
explicitly specified since they depend on the problem at hand.
We discuss these issues in more detail in \Cref{sec:numerics-cont}.
Next, in Line~\ref{line:l2opt-stop},
we use the relative change in the $\Ltwo$ output error as the stopping
criterion, as it only depends on the reduced quantities.
However, one can easily incorporate more sophisticated stopping criteria if
desired.
Finally, the pseudocode does not check for the invertibility of $\Ar(\pp)$.
In our experiments, the obtained \acp{ddrom} are well-conditioned,
which could be explained by the objective function increasing when
$\normF{\Ar(\pp)^{-1}}$ is large.

The computational complexity of the \Cref{alg:l2opt} depends on the size of the
\ac{ddrom};
more specifically on the number of unknowns,
which is $\qAr \nrom^2 + \qBr \nrom \nin + \qCr \nout \nrom$.
Note that this is linear in
the number of terms in the parameter-separable forms,
the number of forcings, and
the number of outputs, but
quadratic in the reduced order.

\section{Numerical Examples for the Continuous \texorpdfstring{$\Ltwo$}{L2}
  Norm}%
\label{sec:numerics-cont}
Here we present numerical experiments demonstrating the performance of
$\Ltwo$-Opt-PSF (\Cref{alg:l2opt}) when $\measure$ is the Lebesgue measure
(thus, a continuous least-squares problem).

In particular, we focus on the case of a one-dimensional parameter,
i.e., $\npar = 1$,
where we use \texttt{quad} and \texttt{quad\_vec} from \texttt{scipy.integrate}
(from SciPy~\cite{VirGOetal20}),
methods for numerical integration of scalar and vector-valued
functions, respectively,
to evaluate $\obj$~\eqref{eq:l2error} and
its gradients (see \Cref{thm:gradients}).
For the gradient-based optimization method in \Cref{alg:l2opt},
we chose Broyden-Fletcher-Goldfarb-Shanno~(BFGS)~\cite{NocW99} algorithm.
In \Cref{alg:l2opt} we set $\mathtt{maxit} = 1000$ and $\mathtt{tol} = 10^{-6}$
for all examples.
For evaluating $\Linf$ errors,
we use \texttt{scipy.optimize.shgo} (from SciPy~\cite{VirGOetal20}),
a method for global minimization.

We compare \ac{rb}, \ac{pod}, and $\Ltwo$-Opt-PSF.\
Given a training set $\pset_{\textnormal{train}} \subseteq \pset$ and \iac{fom},
\Ac{rb} constructs the \ac{rom} via Galerkin projection,
where the projection basis is built in a greedy manner to reduce the $\Linf$
error.
In particular, we chose to use the strong greedy version given in \Cref{alg:rb},
where the error $\normF{\yf(\pp) - \yr(\pp)}$ is used instead of an error
estimator.
There are different error estimators proposed in the literature
(see, e.g.,~\cite{HesRS16,QuaMN16,FenB19,CheJN19}),
But our focus here is on minimizing the error and less on an efficient
implementation.
\begin{algorithm}
  \caption{Strong Greedy Algorithm}\label{alg:rb}
  \begin{algorithmic}[1]
    \Require%
      \Ac{fom} $(\Af, \Bf, \Cf)$~\eqref{eq:fom} in parameter-separable
      form~\eqref{eq:fom-par-sep-form},
      finite training set $\pset_{\textnormal{train}} \subseteq \pset$,
      maximum reduced order $\nrom_{\max}$,
      tolerance $\mathtt{tol} > 0$.
    \Ensure%
      \Ac{rom} $(\Ar, \Br, \Cr)$.
    \State%
      $V = [\,]$.
    \For{$i$ in $1, 2, \ldots, \nrom_{\max}$}
      \State%
        Find $\pp_{\max} \in \pset_{\textnormal{train}}$ that maximizes the
        error $e(\pp) = \normF{\yf(\pp) - \yr(\pp)}$.
      \If{$e(\pp_{\max}) \le \mathtt{tol}$}
        \State%
          Exit the \textbf{for} loop.
      \EndIf%
      \State%
        Set $V = \begin{bmatrix} V & \xf(\pp_{\max}) \end{bmatrix}$ and
        orthonormalize it.
      \State%
        Form \iac{rom} $(\Ar, \Br, \Cr)$ with $\cAr_i, \cBr_j, \cCr_k$ as
        in~\eqref{eq:galerkin}.
    \EndFor%
    \State%
      Return $(\Ar, \Br, \Cr)$.
  \end{algorithmic}
\end{algorithm}
\Ac{pod}~\cite{BenGQetal20b} is also based on a Galerkin projection and
tries to find a subspace that approximates the solution set
$\{\xf(\pp) : \pp \in \pset\} \subseteq \CCf$ in the $\Ltwo$-optimal sense
given a training set.
The pseudocode for \ac{pod} is given in \Cref{alg:pod}.
\begin{algorithm}[tb]
  \caption{Proper Orthogonal Decomposition}\label{alg:pod}
  \begin{algorithmic}[1]
    \Require%
      \Ac{fom} $(\Af, \Bf, \Cf)$~\eqref{eq:fom} in parameter-separable
      form~\eqref{eq:fom-par-sep-form},
      finite training set $\pset_{\textnormal{train}} \subseteq \pset$,
      reduced order $\nrom$.
    \Ensure%
      \Ac{rom} $(\Ar, \Br, \Cr)$.
    \State%
      $X =
      \begin{bmatrix}
        \xf(\pp_1) & \xf(\pp_2) & \cdots
        & \xf(\pp_{\card{\pset_{\textnormal{train}}}})
      \end{bmatrix}$.
    \State%
      Compute the singular value decomposition $X = U \Sigma W\tran$.
    \State%
      Set $V$ as the first $\nrom$ columns of $U$.
    \State%
      Form \iac{rom} $(\Ar, \Br, \Cr)$ with $\cAr_i, \cBr_j, \cCr_k$ as
      in~\eqref{eq:galerkin}.
    \State%
      Return $(\Ar, \Br, \Cr)$.
  \end{algorithmic}
\end{algorithm}

We focus on stationary parametric \acp{pde} as numerical examples.
As mentioned in~\Cref{rem:future-work},
the various implications of our approach for the dynamical systems case,
specifically for $\Htwo$ and $\HtwoLtwo$-optimal model order reduction of
\ac{lti} systems,
are presented in detail in~\cite{MliG22b}.
In \Cref{sec:numerics-disc},
where we consider a discrete $\Ltwo$ norm,
we include \iac{lti} example.

In the following numerical examples,
the parameter space $\pset$ is a segment $[a, b] \subset \RR$ and
we chose $\pset_{\textnormal{train}} = \mathtt{linspace}(a, b, 100)$ for \ac{rb}
and \ac{pod}
where \texttt{linspace} refers to the NumPy (\cite{HarMWetal20}) method
\texttt{numpy.linspace} returning a vector of equidistant points in $[a, b]$
including the boundaries.
Although we describe the \acp{fom} we use in every example,
we note that $\Ltwo$-Opt-PSF only needs access to the output of the \ac{fom} and
not its state or matrices.
The \ac{fom} description is given since \ac{rb} and \ac{pod} are
projection-based and work with the full-order operators.

\begin{mdframed}[
    skipabove=1ex,
    skipbelow=1ex,
    innerleftmargin=1ex,
    innerrightmargin=1ex,
    innertopmargin=1ex,
    innerbottommargin=1ex,
  ]
  \textbf{Code Availability:}
  The source code used to compute the presented results can be obtained
  from~\cite{Mli22}.
  The code was written in the Python programming language using
  pyMOR~\cite{MilRS16}.
\end{mdframed}

\subsection{Poisson Equation}
We consider the Poisson equation over the unit square $\Omega = {(0, 1)}^2$ with
homogeneous Dirichlet boundary conditions:
\begin{subequations}\label{eq:poisson}
\begin{align}
  -\nabla \cdot (d(\xi, \pp) \nabla \xf(\xi, \pp)) & = 1,
  & \xi \in \Omega, \\*
  \xf(\xi, \pp) & = 0,
  & \xi \in \partial \Omega,
\end{align}
\end{subequations}
where
$d(\xi, \pp) = \xi_1 + \pp (1 - \xi_1)$ and
$\pset = [0.1, 10]$.
After a finite element discretization,
we obtain \iac{fom} of the form~\eqref{eq:stationary}
with $\nfom = 1089$ and $\nin = 1$.
For the output,
we chose $\cCf = \cBf\tran$ (thus, $\nout = 1$).
Solutions for a few parameter values are given in \Cref{fig:poisson-solutions}.
The output can be seen in the left plot in \Cref{fig:poisson-outputs}.
\begin{figure}[tb]
  \centering
  \begin{minipage}{0.3\linewidth}
    \centering
    \includegraphics[width=\linewidth]{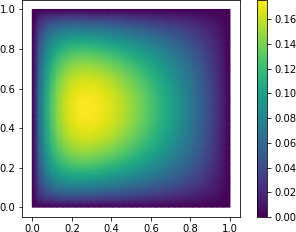} \\
    {\footnotesize $\pp = 0.1$}
  \end{minipage}
  \hfill
  \begin{minipage}{0.3\linewidth}
    \centering
    \includegraphics[width=\linewidth]{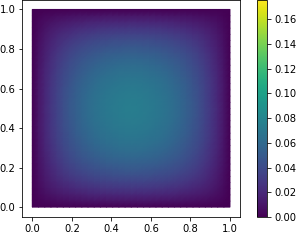} \\
    {\footnotesize $\pp = 1$}
  \end{minipage}
  \hfill
  \begin{minipage}{0.3\linewidth}
    \centering
    \includegraphics[width=\linewidth]{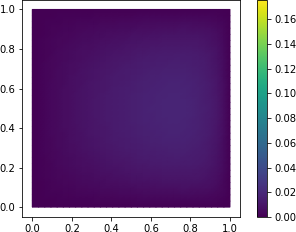} \\
    {\footnotesize $\pp = 10$}
  \end{minipage}
  \caption{Poisson example \ac{fom} solutions for different parameter values}%
  \label{fig:poisson-solutions}
\end{figure}
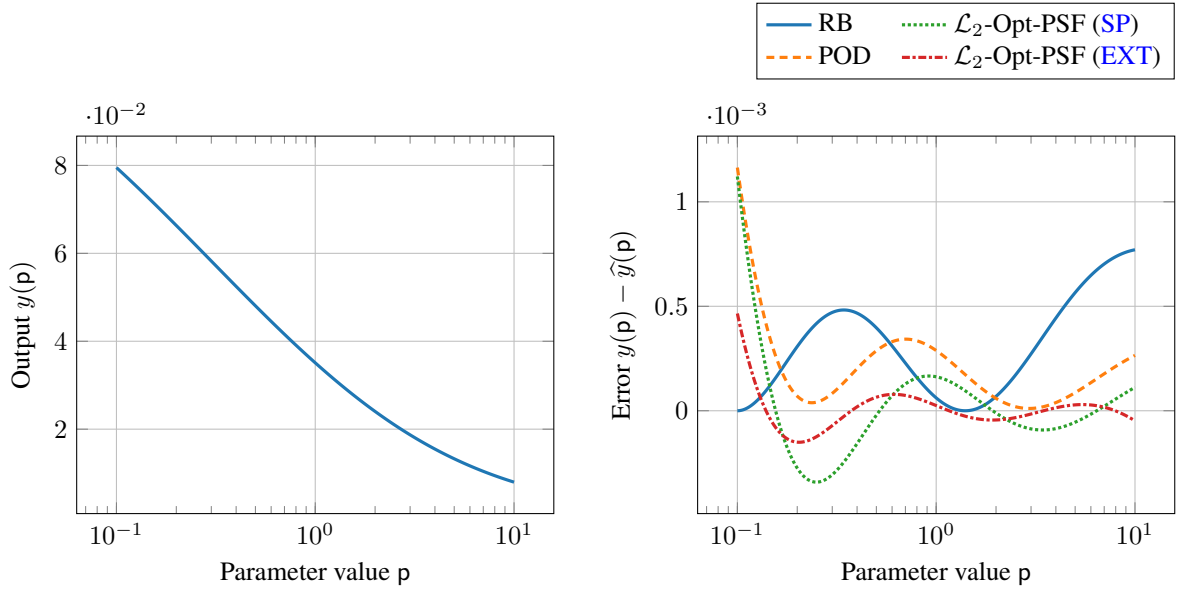
\begin{figure}[tb]
  \centering
  \tikzsetnextfilename{poisson-outputs}
  \begin{tikzpicture}
    \begin{axis}[
        width=0.48\linewidth,
        height=0.4\linewidth,
        xlabel={Parameter value $\pp$},
        ylabel={Output $\yf(\pp)$},
        xmode=log,
        grid=major,
        cycle list name=mpl,
      ]
      \addplot table [x=p, y=y] {poisson_output.txt};
    \end{axis}
    \begin{axis}[
        at={(0.5\linewidth, 0)},
        width=0.48\linewidth,
        height=0.4\linewidth,
        xlabel={Parameter value $\pp$},
        ylabel={Error $\yf(\pp) - \yr(\pp)$},
        xmode=log,
        grid=major,
        legend entries={RB, POD, $\Ltwo$-Opt-PSF~\eqref{eq:poisson-rom-sp},
          $\Ltwo$-Opt-PSF~\eqref{eq:poisson-rom-ext}},
        legend columns=2,
        legend style={
          at={(1, 1.15)},
          anchor=south east,
          /tikz/every even column/.append style={column sep=2ex},
        },
        legend cell align={left},
        transpose legend,
        cycle list name=mpl,
      ]
      \addplot table [x=p, y=rb] {poisson_output_errors.txt};
      \addplot table [x=p, y=pod] {poisson_output_errors.txt};
      \addplot table [x=p, y=l2] {poisson_output_errors.txt};
      \addplot table [x=p, y=l2ext] {poisson_output_errors.txt};
    \end{axis}
  \end{tikzpicture}
  \caption{Poisson example \ac{fom} output and
    pointwise output errors for \acp{rom} of order $2$}%
  \label{fig:poisson-outputs}
\end{figure}

For the \acp{rom}, we consider a structure-preserving (physics-inspired)
\ac{ddrom}
\begin{equation}\tag{SP}\label{eq:poisson-rom-sp}
  \begin{aligned}
    \myparen*{\cAr_1 + \pp \cAr_2} \xr(\pp) & = \cBr, \\*
    \yr(\pp) & = \cCr \xr(\pp),
  \end{aligned}
\end{equation}
and an extended version
\begin{equation}\tag{EXT}\label{eq:poisson-rom-ext}
  \begin{aligned}
    \myparen*{\cAr_1 + \pp \cAr_2} \xr(\pp) & = \cBr_1 + \pp \cBr_2, \\*
    \yr(\pp) & = \myparen*{\cCr_1 + \pp \cCr_2} \xr(\pp).
  \end{aligned}
\end{equation}
We include the extended version to highlight that the proposed approach offers
flexibility to include additional structures that are not necessarily present in
the \ac{fom}
(further emphasizing that the \ac{fom} is not needed,
but only the output samples).
We also note that~\eqref{eq:poisson-rom-ext} cannot be obtained via a
state-independent linear projection.

We choose the reduced order $\nrom = 2$.
The right plot in \Cref{fig:poisson-outputs} shows the output errors,
resulting from
\ac{rb},
\ac{pod},
$\Ltwo$-Opt-PSF with structure preservation
(initialized using \ac{pod}),
and $\Ltwo$-Opt-PSF with extended form
(initialized using the result of structure-preserving $\Ltwo$-Opt-PSF).
The relative $\Ltwo$ errors are, respectively,
$2.5577 \times 10^{-2}$,
$8.2483 \times 10^{-3}$,
$4.3826 \times 10^{-3}$, and
$1.6468 \times 10^{-3}$,
illustrating that $\Ltwo$-Opt-PSF produces the smallest $\Ltwo$ error.
The relative $\Linf$ errors are, respectively,
$9.6919 \times 10^{-3}$,
$1.4639 \times 10^{-2}$,
$1.4104 \times 10^{-2}$, and
$5.8541 \times 10^{-3}$.
We observe that \ac{rb} and \ac{pod} produce \acp{rom} with nonnegative output
error.
This is a general property of Galerkin projection applied to systems with
symmetric positive definite $\Af(\pp)$ and $\Bf(\pp) = \Cf(\pp)\tran$
(see, e.g., Theorem~1.4 in~\cite{BenGQetal20b}).
We also observe that even though $\Ltwo$-Opt-PSF preserved the symmetry
properties in the \acp{ddrom}, without enforcing it explicitly,
it did not give \acp{ddrom} that has nonnegative output error.
This implies that $\Ltwo$-Opt-PSF \acp{ddrom} are not based on Galerkin
projection.
An explanation for why symmetry is preserved in $\Ltwo$-Opt-PSF without
explicitly enforcing it is that the gradients of the objective function
(\Cref{thm:gradients}) have the same symmetry properties.

The relative $\Ltwo$ and $\Linf$ errors are shown in
\Cref{fig:poisson-output-rom-errors} for a range of reduced orders.
We observe that $\Ltwo$-Opt-PSF~\eqref{eq:poisson-rom-sp} and
$\Ltwo$-Opt-PSF~\eqref{eq:poisson-rom-ext} have consistently lower $\Ltwo$ error
compared to \ac{pod} and \ac{rb}.
Moreover, $\Ltwo$-Opt-PSF~\eqref{eq:poisson-rom-sp} and
$\Ltwo$-Opt-PSF~\eqref{eq:poisson-rom-ext} have comparable $\Linf$ errors to
(and for some $\nrom$ values even smaller than)
\ac{rb} and \ac{pod} even though not optimized for this error measure.
\begin{figure}[tb]
  \centering
  \tikzsetnextfilename{poisson-rom-errors}
  \begin{tikzpicture}
    \begin{semilogyaxis}[
        width=0.48\linewidth,
        height=0.4\linewidth,
        xlabel={Reduced order},
        ylabel={Relative $\Ltwo$ error},
        grid=major,
        legend entries={RB, POD, $\Ltwo$-Opt-PSF~\eqref{eq:poisson-rom-sp},
          $\Ltwo$-Opt-PSF~\eqref{eq:poisson-rom-ext}},
        legend columns=-1,
        legend style={
          at={(1.15, 1.1)},
          anchor=south,
          /tikz/every even column/.append style={column sep=2ex},
        },
        legend cell align={left},
        xtick={1, 2, 3, 4, 5},
        cycle list name=mplmark,
      ]
      \addplot table [x=r, y=rb] {poisson_output_rom_l2_errors.txt};
      \addplot table [x=r, y=pod] {poisson_output_rom_l2_errors.txt};
      \addplot table [x=r, y=l2] {poisson_output_rom_l2_errors.txt};
      \addplot table [x=r, y=l2ext] {poisson_output_rom_l2_errors.txt};
    \end{semilogyaxis}
    \begin{semilogyaxis}[
        at={(0.5\linewidth, 0)},
        width=0.48\linewidth,
        height=0.4\linewidth,
        xlabel={Reduced order},
        ylabel={Relative $\Linf$ error},
        grid=major,
        xtick={1, 2, 3, 4, 5},
        cycle list name=mplmark,
      ]
      \addplot table [x=r, y=rb] {poisson_output_rom_linf_errors.txt};
      \addplot table [x=r, y=pod] {poisson_output_rom_linf_errors.txt};
      \addplot table [x=r, y=l2] {poisson_output_rom_linf_errors.txt};
      \addplot table [x=r, y=l2ext] {poisson_output_rom_linf_errors.txt};
    \end{semilogyaxis}
  \end{tikzpicture}
  \caption{Poisson example errors for \acp{rom} of different orders}%
  \label{fig:poisson-output-rom-errors}
\end{figure}
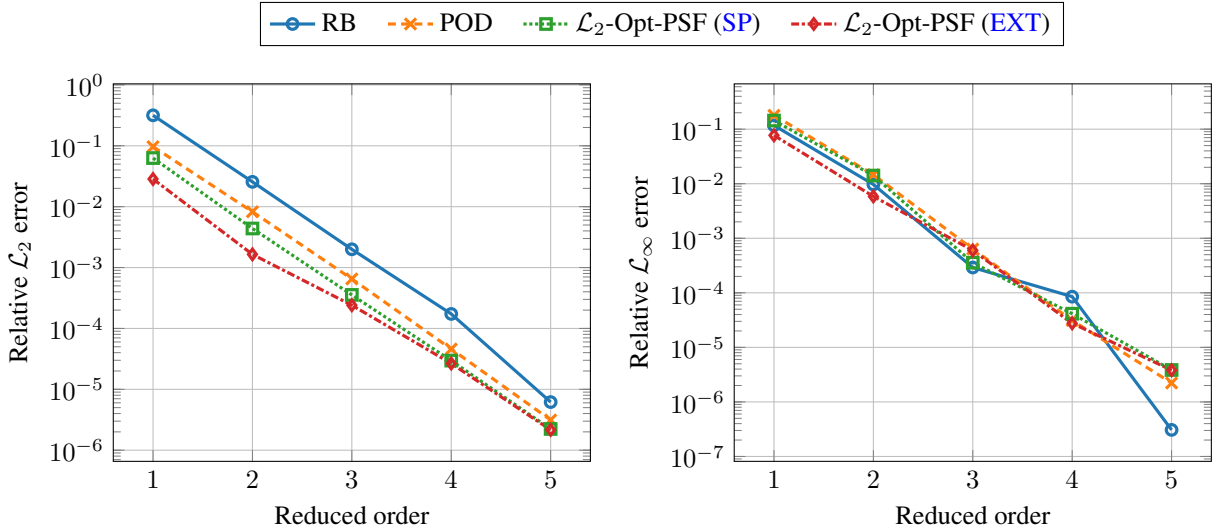

\subsection{Non-separable Example}
We modify the Poisson equation in~\eqref{eq:poisson} by using a diffusion term
that is not parameter-separable.
Specifically, we take
\begin{equation*}
  d(\xi, \pp)
  = 1 - \frac{9}{10} e^{-5 \myparen*{{(\xi_1 - \pp)}^2 + {(\xi_2 - \pp)}^2}}
\end{equation*}
and set
$\pset = [0, 1]$.
After a finite element discretization,
the \ac{fom} is of the form
\begin{align*}
  \Af(\pp) \xf(\pp) & = \cBf, \\*
  \yf(\pp) & = \cCf \xf(\pp),
\end{align*}
with $\nfom = 1089$ and $\nin = 1$, and
$\Af(\pp)$ needs to be assembled for every new parameter value $\pp$.
For the output,
we again chose $\cCf = \cBf\tran$ ($\nout = 1$).
Solutions for a few parameter values are given in \Cref{fig:nonsep-solutions}.
The output can be seen in the left plot in \Cref{fig:nonsep-outputs}.
\begin{figure}[tb]
  \centering
  \begin{minipage}{0.3\linewidth}
    \centering
    \includegraphics[width=\linewidth]{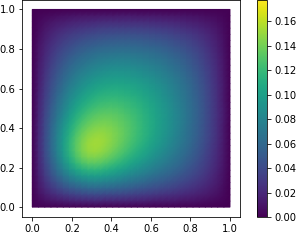} \\
    {\footnotesize $\pp = 0.25$}
  \end{minipage}
  \hfill
  \begin{minipage}{0.3\linewidth}
    \centering
    \includegraphics[width=\linewidth]{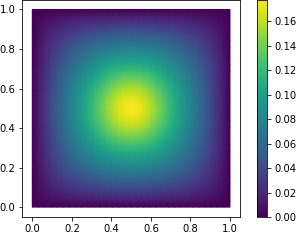} \\
    {\footnotesize $\pp = 0.5$}
  \end{minipage}
  \hfill
  \begin{minipage}{0.3\linewidth}
    \centering
    \includegraphics[width=\linewidth]{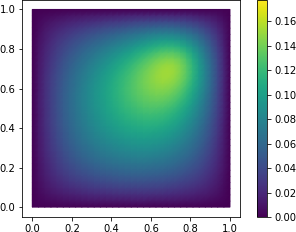} \\
    {\footnotesize $\pp = 0.75$}
  \end{minipage}
  \caption{Poisson example \ac{fom} solutions for different parameter values}%
  \label{fig:nonsep-solutions}
\end{figure}
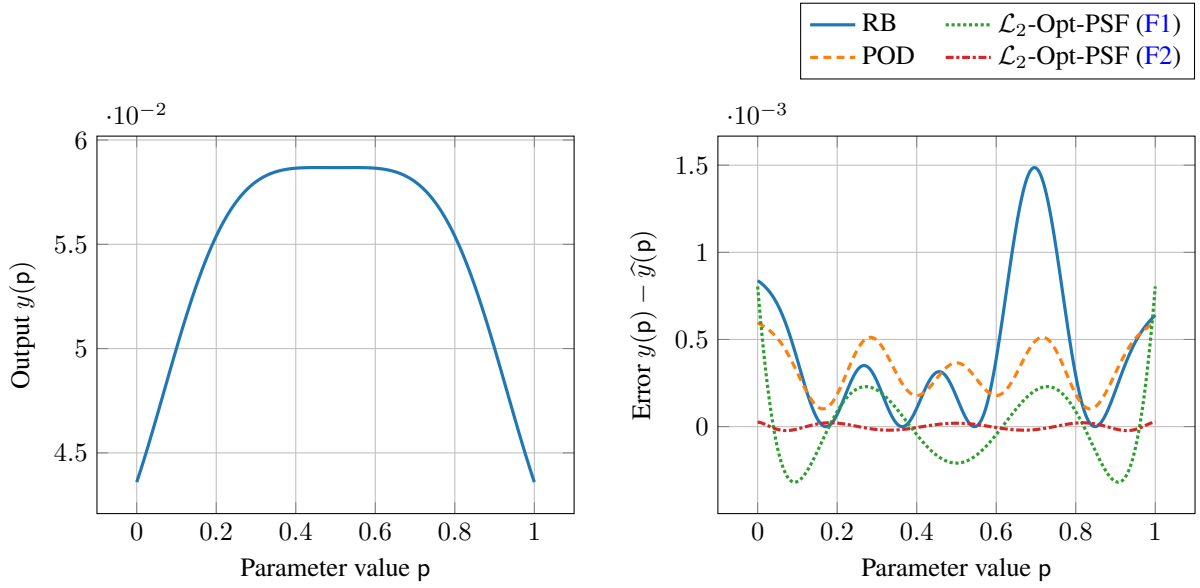
\begin{figure}[tb]
  \centering
  \tikzsetnextfilename{nonsep-outputs}
  \begin{tikzpicture}
    \begin{axis}[
        width=0.48\linewidth,
        height=0.4\linewidth,
        xlabel={Parameter value $\pp$},
        ylabel={Output $\yf(\pp)$},
        grid=major,
        cycle list name=mpl,
      ]
      \addplot table [x=p, y=y] {nonsep_output.txt};
    \end{axis}
    \begin{axis}[
        at={(0.5\linewidth, 0)},
        width=0.48\linewidth,
        height=0.4\linewidth,
        xlabel={Parameter value $\pp$},
        ylabel={Error $\yf(\pp) - \yr(\pp)$},
        grid=major,
        legend entries={RB, POD, $\Ltwo$-Opt-PSF~\eqref{eq:nonsep-rom-1},
          $\Ltwo$-Opt-PSF~\eqref{eq:nonsep-rom-2}},
        legend columns=2,
        legend style={
          at={(1, 1.15)},
          anchor=south east,
          /tikz/every even column/.append style={column sep=2ex},
        },
        legend cell align={left},
        transpose legend,
        cycle list name=mpl,
      ]
      \addplot table [x=p, y=rb] {nonsep_errors.txt};
      \addplot table [x=p, y=pod] {nonsep_errors.txt};
      \addplot table [x=p, y=l2] {nonsep_errors.txt};
      \addplot table [x=p, y=l2ext] {nonsep_errors.txt};
    \end{axis}
  \end{tikzpicture}
  \caption{Non-separable example \ac{fom} output and
    pointwise output errors for \acp{rom} of order $2$}%
  \label{fig:nonsep-outputs}
\end{figure}
Using \ac{rb} or \ac{pod} produces \acp{rom} of the form
\begin{align*}
  V\tran \Af(\pp) V \xr(\pp) & = V\tran \cBf, \\*
  \yr(\pp) & = \cCf V \xr(\pp).
\end{align*}
Notably, an efficient computation of $V\tran \Af(\pp) V$ requires a further
approximation of $\Af(\pp)$ in a parameter-separable form, e.g.,
using the empirical interpolation method~\cite{BarMNP04}
as mentioned in~\Cref{sec:preliminaries};
for details, see, e.g.,~\cite{BenGW15}.
However, in order not to
degrade the accuracy of \ac{rb} and \ac{pod} models, we skip that step.
For the $\Ltwo$-optimal \acp{ddrom}, we consider two forms:
\begin{equation}\tag{F1}\label{eq:nonsep-rom-1}
  \begin{aligned}
    \myparen*{
    \cAr_1
    + \myparen*{\pp - \tfrac{1}{2}}^2 \cAr_2
    + \myparen*{\pp - \tfrac{1}{2}}^4 \cAr_3
    } \xr(\pp)
    & =
      \cBr, \\*
    \yr(\pp)
    & =
      \cCr \xr(\pp),
  \end{aligned}
\end{equation}
and
\begin{equation}\tag{F2}\label{eq:nonsep-rom-2}
  \begin{aligned}
    & \hspace{-3em}
    \myparen*{
      \cAr_1
      + \myparen*{\pp - \tfrac{1}{2}}^2 \cAr_2
      + \myparen*{\pp - \tfrac{1}{2}}^4 \cAr_3
      + \myparen*{\pp - \tfrac{1}{2}}^6 \cAr_4
    } \xr(\pp) \\*
    & =
      \cBr_1
      + \myparen*{\pp - \tfrac{1}{2}}^2 \cBr_2
      + \myparen*{\pp - \tfrac{1}{2}}^4 \cBr_3
      + \myparen*{\pp - \tfrac{1}{2}}^6 \cBr_4, \\*
    \yr(\pp)
    & =
      \myparen*{
        \cCr_1
        + \myparen*{\pp - \tfrac{1}{2}}^2 \cCr_2
        + \myparen*{\pp - \tfrac{1}{2}}^4 \cCr_3
        + \myparen*{\pp - \tfrac{1}{2}}^6 \cCr_4
      } \xr(\pp),
  \end{aligned}
\end{equation}
which exploit the symmetry in $\yf(\pp)$ around $\pp = \frac{1}{2}$.
The right plot in \Cref{fig:nonsep-outputs} shows the output errors of
\acp{rom} of order $4$,
resulting from
\ac{rb},
\ac{pod}, and
both $\Ltwo$-Opt-PSF models
(both initialized with
$\cAr_1 = I$,
$\cBr_1 = \mathbf{1}$,
$\cCr_1 = \mathbf{1}\tran$, and
$\cAr_i = 0$,
$\cBr_i = 0$, and
$\cCr_i = 0$
for $i \ge 2$,
where $\mathbf{1}$ is the vector of all ones).
We observe that the \ac{rb} and \ac{pod} again produce \acp{rom} with
nonnegative output error.
The relative $\Ltwo$ and $\Linf$ errors listed in \Cref{tab:nonsep-errors}
show significant improvements in $\Ltwo$ error minimization via $\Ltwo$-Opt-PSF,
especially for the second \ac{ddrom} form.
\begin{table}[tb]
  \centering
  \caption{Relative errors for the non-separable example}%
  \label{tab:nonsep-errors}
  \begin{tabular}{ccccc}
    Error measure & RB & POD & $\Ltwo$-Opt-PSF 1 & $\Ltwo$-Opt-PSF 2 \\
    \midrule
    $\Ltwo$
                  & $1.0441 \times 10^{-2}$
                       & $6.4925 \times 10^{-3}$
                             & $3.8323 \times 10^{-3}$
                                           & $2.7439 \times 10^{-4}$ \\
    $\Linf$
                  & $1.4279 \times 10^{-2}$
                       & $1.0146 \times 10^{-2}$
                             & $1.3707 \times 10^{-2}$
                                           & $4.4095 \times 10^{-4}$
  \end{tabular}
\end{table}

\subsection{Convection-diffusion Problem}
Consider the convection-diffusion equation on the unit square
$\Omega = {(0, 1)}^2$ with homogeneous Dirichlet boundary conditions:
\begin{align*}
  \nabla \cdot ((\cos{\pp}, \sin{\pp}) \xf(\xi, \pp))
  -\nabla \cdot (d \nabla \xf(\xi, \pp))
  & = 1,
  & \xi \in \Omega, \\*
  \xf(\xi, \pp) & = 0,
  & \xi \in \partial \Omega,
\end{align*}
where $d = 2^{-5}$ and $\pset = [0, 2 \pi]$.
For the outputs, we chose
$\yf_{\ell}(\pp) = \int_{\Omega_{\ell}} \xf(\xi, \pp) \dif{\xi}$,
$\ell = 1, 2, 3, 4$,
where
$\Omega_1 = (0, \frac{1}{2})^2$,
$\Omega_2 = (\frac{1}{2}, 1) \times (0, \frac{1}{2})$,
$\Omega_3 = (\frac{1}{2}, 1)^2$,
$\Omega_4 = (0, \frac{1}{2}) \times (\frac{1}{2}, 1)$.
After a finite element discretization,
we obtain \iac{fom}
\begin{align*}
  (\cAf_1 + \cos(\pp) \cAf_2 + \sin(\pp) \cAf_3) \xf(\pp) & = \cBf, \\*
  \yf(\pp) & = \cCf \xf(\pp),
\end{align*}
with $\nfom = 1089$, $\nin = 1$, and $\nout = 4$.
Solutions for a few parameter values are given in
\Cref{fig:convection-solutions}.
\begin{figure}[tb]
  \centering
  \begin{minipage}{0.3\linewidth}
    \centering
    \includegraphics[width=\linewidth]{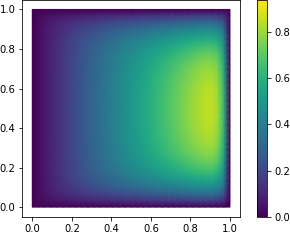} \\
    {\footnotesize $\pp = 0$}
  \end{minipage}
  \hfill
  \begin{minipage}{0.3\linewidth}
    \centering
    \includegraphics[width=\linewidth]{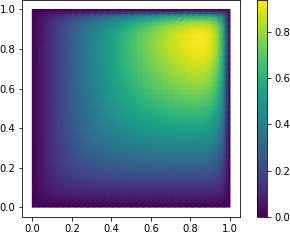} \\
    {\footnotesize $\pp = \frac{\pi}{4}$}
  \end{minipage}
  \hfill
  \begin{minipage}{0.3\linewidth}
    \centering
    \includegraphics[width=\linewidth]{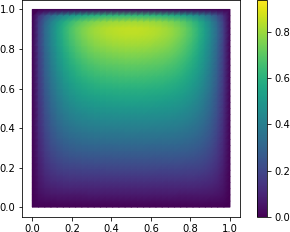} \\
    {\footnotesize $\pp = \frac{\pi}{2}$}
  \end{minipage}
  \caption{Convection-diffusion example \ac{fom} solutions for different
    parameter values}%
  \label{fig:convection-solutions}
\end{figure}
The left plot in \Cref{fig:convection-outputs} shows the output $\yf$.
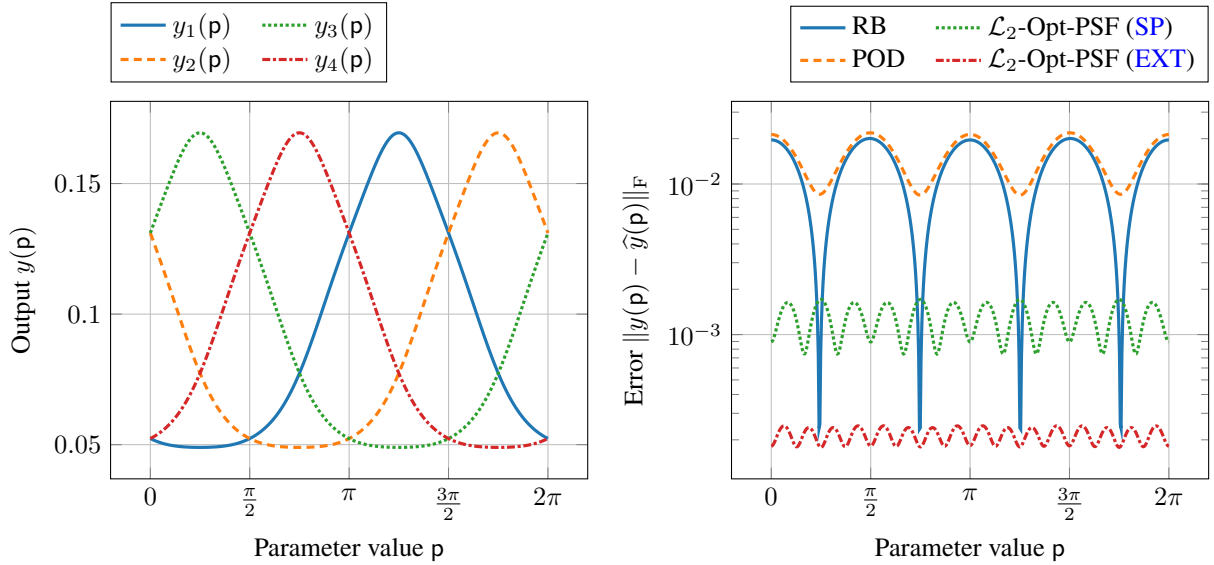
\begin{figure}[tb]
  \centering
  \tikzsetnextfilename{convection-output}
  \begin{tikzpicture}
    \begin{axis}[
        width=0.48\linewidth,
        height=0.4\linewidth,
        xlabel={Parameter value $\pp$},
        ylabel={Output $\yf(\pp)$},
        grid=major,
        legend entries={$\yf_1(\pp)$, $\yf_2(\pp)$, $\yf_3(\pp)$, $\yf_4(\pp)$},
        legend columns=2,
        legend style={
          at={(0, 1.05)},
          anchor=south west,
          /tikz/every even column/.append style={column sep=2ex},
        },
        legend cell align={left},
        transpose legend,
        y tick label style={
          /pgf/number format/.cd,
          fixed relative,
        },
        cycle list name=mpl,
        xtick={0, 1.5708, 3.14159, 4.7123889, 6.28318},
        xticklabels={$0$, $\frac{\pi}{2}$, $\pi$, $\frac{3 \pi}{2}$, $2 \pi$},
      ]
      \addplot table [x=p, y=y1] {convection_output.txt};
      \addplot table [x=p, y=y2] {convection_output.txt};
      \addplot table [x=p, y=y3] {convection_output.txt};
      \addplot table [x=p, y=y4] {convection_output.txt};
    \end{axis}
    \begin{axis}[
        at={(0.5\linewidth, 0)},
        width=0.48\linewidth,
        height=0.4\linewidth,
        xlabel={Parameter value $\pp$},
        ylabel={Error $\normF{\yf(\pp) - \yr(\pp)}$},
        grid=major,
        legend entries={RB, POD, $\Ltwo$-Opt-PSF~\eqref{eq:conv-rom-sp},
          $\Ltwo$-Opt-PSF~\eqref{eq:conv-rom-ext}},
        legend columns=2,
        legend style={
          at={(1, 1.05)},
          anchor=south east,
          /tikz/every even column/.append style={column sep=2ex},
        },
        legend cell align={left},
        transpose legend,
        cycle list name=mpl,
        ymode=log,
        xtick={0, 1.5708, 3.14159, 4.7123889, 6.28318},
        xticklabels={$0$, $\frac{\pi}{2}$, $\pi$, $\frac{3 \pi}{2}$, $2 \pi$},
      ]
      \addplot table [x=p, y=rb] {convection_output_errors.txt};
      \addplot table [x=p, y=pod] {convection_output_errors.txt};
      \addplot table [x=p, y=l2] {convection_output_errors.txt};
      \addplot table [x=p, y=l2ext] {convection_output_errors.txt};
    \end{axis}
  \end{tikzpicture}
  \caption{Convection-diffusion example \ac{fom} output and
    pointwise output error norms for \acp{rom} of order $4$}%
  \label{fig:convection-outputs}
\end{figure}

As for the first example, for the \acp{ddrom},
we consider a structure-preserving model
\begin{equation}\tag{SP}\label{eq:conv-rom-sp}
  \begin{aligned}
    \myparen*{\cAr_1 + \cos(\pp) \cAr_2 + \sin(\pp) \cAr_3} \xr(\pp)
    & =
      \cBr, \\*
    \yr(\pp)
    & =
      \cCr \xr(\pp),
  \end{aligned}
\end{equation}
and an extended version
\begin{equation}\tag{EXT}\label{eq:conv-rom-ext}
  \begin{aligned}
    \myparen*{\cAr_1 + \cos(\pp) \cAr_2 + \sin(\pp) \cAr_3} \xr(\pp)
    & =
      \cBr_1 + \cos(\pp) \cBr_2 + \sin(\pp) \cBr_3, \\*
    \yr(\pp)
    & =
      \myparen*{\cCr_1 + \cos(\pp) \cCr_2 + \sin(\pp) \cCr_3} \xr(\pp).
  \end{aligned}
\end{equation}
The $\cos(\pp)$ and $\sin(\pp)$ terms are inspired by the periodicity of
$y(\pp)$.
The right plot in \Cref{fig:convection-outputs} shows the output errors of
\acp{rom} of order $4$,
due to
\ac{rb},
\ac{pod},
$\Ltwo$-Opt-PSF with structure preservation
(initialized using \ac{rb}), and
$\Ltwo$-Opt-PSF with extended form
(initialized using the structure-preserving $\Ltwo$-Opt-PSF),
illustrating that both $\Ltwo$-Opt-PSF models significantly outperform \ac{rb}
and \ac{pod}.
The relative $\Ltwo$ and $\Linf$ errors for a range of reduced orders are shown
in \Cref{fig:convection-output-rom-errors}.
We observe significant improvements in both the $\Ltwo$ and $\Linf$ errors in
many cases.
\begin{figure}[tb]
  \centering
  \tikzsetnextfilename{convection-rom-errors}
  \begin{tikzpicture}
    \begin{semilogyaxis}[
        width=0.48\linewidth,
        height=0.4\linewidth,
        xlabel={Reduced order},
        ylabel={Relative $\Ltwo$ error},
        grid=major,
        legend entries={RB, POD, $\Ltwo$-Opt-PSF~\eqref{eq:conv-rom-sp},
          $\Ltwo$-Opt-PSF~\eqref{eq:conv-rom-ext}},
        legend columns=-1,
        legend style={
          at={(1.22, 1.05)},
          anchor=south,
          /tikz/every even column/.append style={column sep=2ex},
        },
        legend cell align={left},
        cycle list name=mplmark,
      ]
      \addplot table [x=r, y=rb] {convection_output_rom_l2_errors.txt};
      \addplot table [x=r, y=pod] {convection_output_rom_l2_errors.txt};
      \addplot table [x=r, y=l2] {convection_output_rom_l2_errors.txt};
      \addplot table [x=r, y=l2ext] {convection_output_rom_l2_errors.txt};
    \end{semilogyaxis}
    \begin{semilogyaxis}[
        at={(0.5\linewidth, 0)},
        width=0.48\linewidth,
        height=0.4\linewidth,
        xlabel={Reduced order},
        ylabel={Relative $\Linf$ error},
        grid=major,
        cycle list name=mplmark,
      ]
      \addplot table [x=r, y=rb] {convection_output_rom_linf_errors.txt};
      \addplot table [x=r, y=pod] {convection_output_rom_linf_errors.txt};
      \addplot table [x=r, y=l2] {convection_output_rom_linf_errors.txt};
      \addplot table [x=r, y=l2ext] {convection_output_rom_linf_errors.txt};
    \end{semilogyaxis}
  \end{tikzpicture}
  \caption{Convection-diffusion example errors for \acp{rom} of different
    orders}%
  \label{fig:convection-output-rom-errors}
\end{figure}
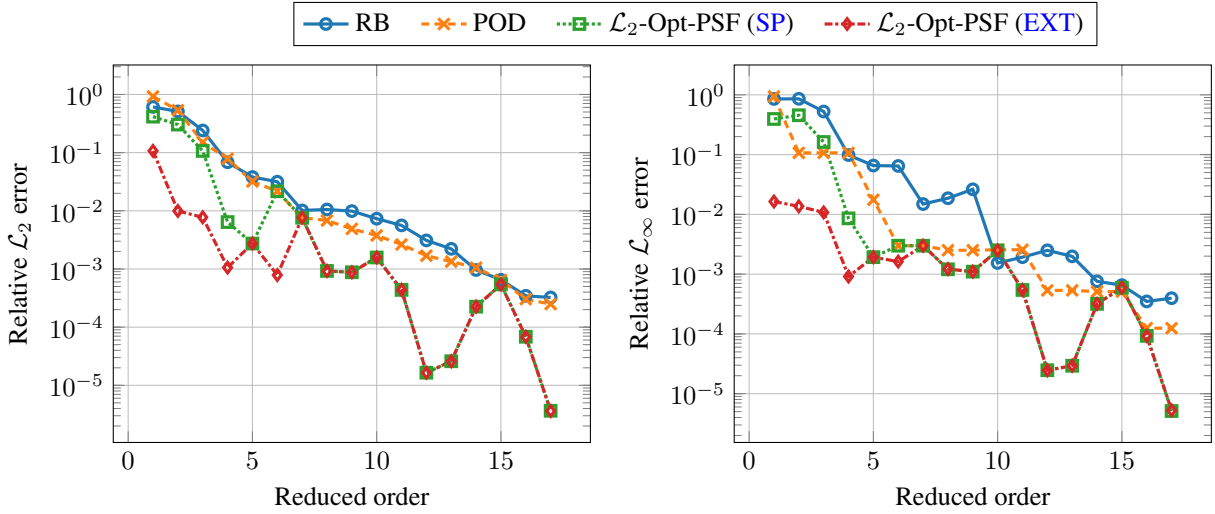

\section{Discrete Least~Squares Norm}%
\label{sec:numerics-disc}
So far, we have considered continuous least-squares problems
where we could evaluate the parameter-to-output mapping $\yf$ for any parameter
value $\pp \in \pset$.
However, in some cases,
we are only given a finite set of points without a chance to re-evaluate $\yf$.
Furthermore, adaptive quadrature used in the previous section can be expensive
for multidimensional parameter spaces.
Thus, it is important to consider the \emph{discrete} setting.

In the discrete setting where we only have samples $(\pp_{\ell}, \yf_{\ell})$,
$\ell = 1, 2, \ldots, N$,
we consider minimizing the \ac{mse}
\begin{equation}\label{eq:mse}
  \objmse
  = \frac{1}{N}
  \sum_{\ell = 1}^N \normF*{\yf_{\ell} - \yr\myparen*{\pp_{\ell}}}^2.
\end{equation}
Our $\Ltwo$-optimal modeling framework recovers the discrete
\ac{mse}~\eqref{eq:mse} by choosing the parameter space as
$\pset = \{\pp_1, \pp_2, \ldots, \pp_N\}$ and the measure as
$\measure = \frac{1}{N} \sum_{\ell = 1}^N \delta_{\pp_{\ell}}$.
With these choices,
the $\Ltwo$ error~\eqref{eq:l2error} becomes the \ac{mse}~\eqref{eq:mse}.
Since the assumptions of \Cref{thm:gradients} are still satisfied for the
discrete measure $\measure$,
we can compute gradients,
which become finite sums in this setting as we summarize in the next result.
\begin{corollary}\label{cor:mse-grad}
  Let
  $(\pp_{\ell}, \yf_{\ell})$, $\ell = 1, 2, \ldots, N$,
  be closed under conjugation and
  $(\cAr_i, \cBr_j, \cCr_k)$ be a tuple of \ac{ddrom} matrices from
  $\romset$~\eqref{eq:setR}.
  Then, the gradients of $\objmse$~\eqref{eq:mse} with respect to the \ac{ddrom}
  matrices are
  \begin{align*}
    \nabla_{\cAr_i} \objmse = {}
    &
      \frac{2}{N}
      \sum_{\ell = 1}^N
      \car_i\myparen*{\overline{\pp_{\ell}}}
      \xrd\myparen*{\pp_{\ell}}
      \mybrack*{\yf_{\ell} - \yr\myparen*{\pp_{\ell}}}
      \xr\myparen*{\pp_{\ell}}\herm,
    & i = 1, 2, \ldots, \qAr, \\*
    \nabla_{\cBr_j} \objmse = {}
    &
      \frac{2}{N}
      \sum_{\ell = 1}^N
      \cbr_j\myparen*{\overline{\pp_{\ell}}}
      \xrd\myparen*{\pp_{\ell}}
      \mybrack*{\yr\myparen*{\pp_{\ell}} - \yf_{\ell}},
    & j = 1, 2, \ldots, \qBr, \\*
    \nabla_{\cCr_k} \objmse = {}
    &
      \frac{2}{N}
      \sum_{\ell = 1}^N
      \ccr_k\myparen*{\overline{\pp_{\ell}}}
      \mybrack*{\yr\myparen*{\pp_{\ell}} - \yf_{\ell}}
      \xr\myparen*{\pp_{\ell}}\herm,
    & k = 1, 2, \ldots, \qCr.
  \end{align*}
\end{corollary}
\begin{proof}
  The result follows directly from \Cref{thm:gradients} by setting
  $\pset = \{\pp_1, \ldots, \pp_N\}$,
  $\measure = \frac{1}{N} \sum_{\ell = 1}^N \delta_{\pp_{\ell}}$, and
  $\yf(\pp_{\ell}) = \yf_{\ell}$.
\end{proof}
For the discrete setting we still use \Cref{alg:l2opt}.
The only difference from the continuous setting is that we do not use adaptive
quadrature, but directly evaluate the sums.
As before, gradient computations only use the output samples $\yf(\pp_{\ell})$.

In the following we demonstrate the results on two examples;
one related to \ac{lti} systems, and
the other arising from a Poisson equation with multiple parameters.

\subsection{Flexible Aircraft Frequency Response Data}
Here we use the data from~\cite{morwiki_flexible_aircraft,PouQV18},
containing samples of a transfer function $H(s)$ of \iac{lti} dynamical system
(as in~\eqref{eq:lti})
describing the influence of wind gust on a flexible aircraft.
In particular, the underlying dynamical system has
$\nin = 1$ forcing (gust disturbance) and
$\nout = 92$ outputs (accelerations and moments at different coordinates of a
flexible aircraft wings and tail);
thus in this problem $H(s) \in \mathbb{C}^{92 \times 1}$.
The output (transfer function) data consists of $N = 421$ pairs
$(\omega_{\ell}, H_{\ell})$, $\ell = 1, 2, \ldots, N$,
where
$\omega_{\ell} > 0$ are the frequencies and
$H_{\ell} \in \CCoi$ are the frequency domain samples.
The left plot in \Cref{fig:flexible-aircraft} shows the magnitudes (norms) of
the frequency responses.
Our goal in this setting is to learn \iac{ddrom} of the form as
in~\eqref{eq:romH}, i.e., $\hH(s) = \hC \myparen{s \hE - \hA}^{-1} \hB$,
that minimizes the MSE distance from the given samples $\{H_\ell\}$.
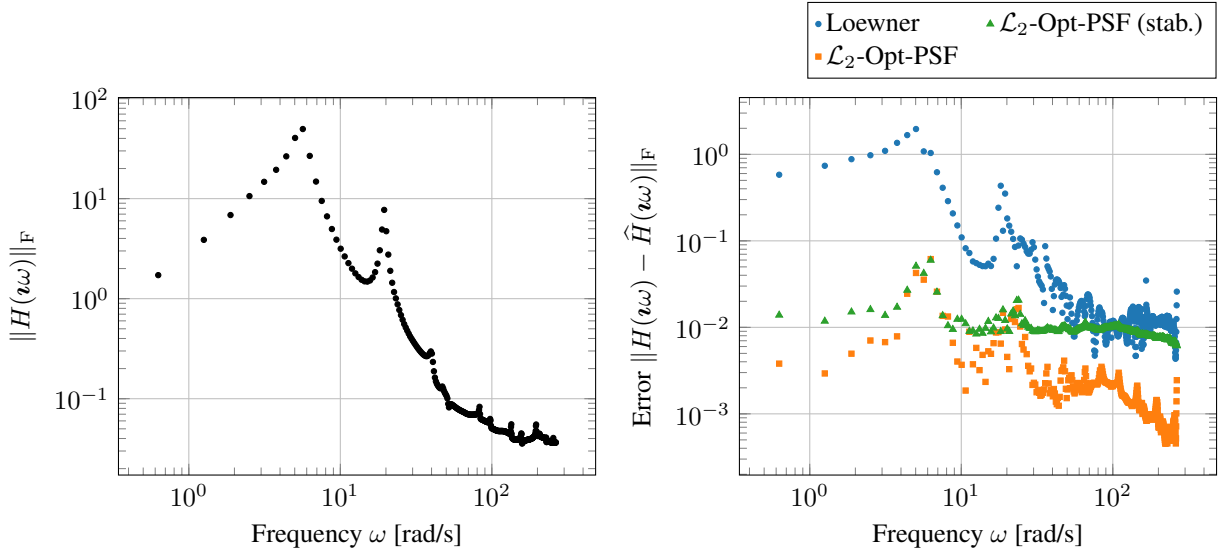
\begin{figure}[tb]
  \centering
  \tikzsetnextfilename{flexible_aircraft-mag}
  \begin{tikzpicture}
    \begin{loglogaxis}[
        width=0.48\linewidth,
        height=0.4\linewidth,
        xlabel={Frequency $\omega$ [rad/s]},
        ylabel={$\normF{H(\imag \omega)}$},
        grid=major,
      ]
      \addplot[black, only marks, mark=*, mark size=1] table [x=w, y=fom]
        {flexible_aircraft_mag.txt};
    \end{loglogaxis}
    \begin{loglogaxis}[
        at={(0.5\linewidth, 0)},
        width=0.48\linewidth,
        height=0.4\linewidth,
        xlabel={Frequency $\omega$ [rad/s]},
        ylabel={Error $\normF{H(\imag \omega) - \hH(\imag \omega)}$},
        grid=major,
        legend entries={Loewner, $\Ltwo$-Opt-PSF, $\Ltwo$-Opt-PSF (stab.)},
        legend columns=2,
        legend style={
          at={(1, 1.05)},
          anchor=south east,
          /tikz/every even column/.append style={column sep=2ex},
        },
        legend cell align={left},
        transpose legend,
        cycle list name=mplmarkonly,
      ]
      \addplot table [x=w, y=rom0]
        {flexible_aircraft_error.txt};
      \addplot table [x=w, y=rom]
        {flexible_aircraft_error.txt};
      \addplot table [x=w, y=romstab] {flexible_aircraft_error.txt};
    \end{loglogaxis}
  \end{tikzpicture}
  \caption{Flexible aircraft example data and
    pointwise output errors for \acp{rom} of order $100$ ($97$)}%
  \label{fig:flexible-aircraft}
\end{figure}

To fulfill the assumptions of \Cref{cor:mse-grad},
we set
\begin{equation*}
  \mybrace*{(\pp_{\ell}, \yf_{\ell})}_{\ell = 1}^{2 N}
  =
  \mybrace*{\myparen*{\imag \omega_{\ell}, H_{\ell}}}_{\ell = 1}^N
  \cup
  \mybrace*{\myparen*{-\imag \omega_{\ell}, \overline{H_{\ell}}}}_{\ell = 1}^N.
\end{equation*}
Based on this setup, we run $\Ltwo$-Opt-PSF with L-BFGS to minimize the
\ac{mse}~\eqref{eq:mse} using the
optimization variables $\hE, \hA \in \CCrr$, $\hB \in \CCri$, and
$\hC \in \CCor$.
We initialize $\Ltwo$-Opt-PSF using the \ac{rom} of order $\nrom = 100$
from~\cite{PouQV18},
which uses {the same data} but employs the interpolatory Loewner
framework~\cite{AntBG20}.
We obtain \iac{ddrom} of order $\nrom = 100$ with $3$ unstable poles.
Then, we project the resulting \ac{ddrom} to its asymptotically stable part of
order $\nrom = 97$.
\Cref{fig:flexible-aircraft} shows the error resulting from
the Loewner-based model from~\cite{PouQV18},
the unstable \ac{ddrom} of order $100$, and
its asymptotically stable part of order $97$.
The respective relative $\Ltwo$ errors are
$4.82 \times 10^{-2}$,
$1.3146 \times 10^{-3}$, and
$2.5255 \times 10^{-3}$.
These results illustrate that $\Ltwo$-Opt-PSF has resulted in more than one
order of magnitude improvement.
Moreover, the unstable part of the $\Ltwo$-Opt-PSF model was small enough that
projecting it onto the asymptotically stable part did not significantly degrade
the accuracy.
Note that we could impose stability constraint directly in the optimization
routine, similar to~\cite{HunMMS21}.
Here we skipped that step to illustrate that $\Ltwo$-Opt-PSF followed by
projection to the asymptotically stable part can still produce \iac{ddrom} with
a small error.
\begin{remark}
  In this example,
  the \ac{ddrom} resulting from $\Ltwo$-Opt-PSF is a rational function;
  thus our $\Ltwo$-optimal modeling framework in this special case has solved
  the rational least-squares problem via a gradient-based descent algorithm.
  Minimizing the \ac{mse} using a rational function
  (rational least-square fitting)
  is an important problem in data-driven modeling of dynamical systems and
  various techniques exist, see,
  e.g.,~\cite{SanK63,GusS99,DrmGB15,DrmGB15a,HokM20,NakST18,BerG17},
  and the references therein.
  In a future work where we specifically focus on dynamical systems and
  approximation of transfer functions from data,
  we will provide more details in this direction.
\end{remark}

\subsection{Thermal Block Example}
We consider a $2 \times 2$ thermal block example over the unit square
$\Omega = {(0, 1)}^2$, i.e.,
we modify the Poisson equation in~\eqref{eq:poisson} by using the diffusion term
\begin{equation*}
  d(\xi, \pp)
  =
  \pp_1 \chi_{\myparen*{0, \frac{1}{2}}^2}(\xi)
  + \pp_2 \chi_{\myparen*{\frac{1}{2}, 1} \times \myparen*{0, \frac{1}{2}}}(\xi)
  + \pp_3 \chi_{\myparen*{\frac{1}{2}, 1}^2}(\xi)
  + \pp_4 \chi_{\myparen*{0, \frac{1}{2}} \times \myparen*{\frac{1}{2}, 1}}(\xi)
\end{equation*}
and
$\pset = {[0.1, 10]}^4$.
Here, $\chi_S$ denotes the characteristic function of the set $S$, i.e.,
$\chi_S(\xi) = 1$ if $\xi \in S$, otherwise $\chi_S(\xi) = 0$.
Therefore, the parameters $\pp_1, \ldots, \pp_4$ represent the diffusivity over
each of the four blocks.
After a finite element discretization,
we obtain the \ac{fom} is of the form
\begin{align*}
  \myparen*{
    \cAf_0
    + \pp_1 \cAf_1
    + \pp_2 \cAf_2
    + \pp_3 \cAf_3
    + \pp_4 \cAf_4
  }
  \xf(\pp)
  & = \cBf, \\*
  \yf(\pp)
  & = \cCf \xf(\pp),
\end{align*}
with $\nfom = 1089$ and $\nin = 1$.
For the output,
we kept $\cCf = \cBf\tran$.
The left plot in \Cref{fig:thermalblock} shows the output of the \ac{fom} over
the equidistant grid
$\pset_{\textnormal{test}} = {\mathtt{linspace}(0.1, 10, 5)}^4$.

\begin{figure}[tb]
  \centering
  \tikzsetnextfilename{thermalblock-output}
  \begin{tikzpicture}
    \begin{axis}[
        width=0.48\linewidth,
        height=0.4\linewidth,
        xlabel={Parameter Index $\ell$},
        ylabel={Output $\yf(\pp_{\ell})$},
        grid=major,
        cycle list name=mplmarkonly,
      ]
      \addplot table [x=p, y=y] {thermalblock_output.txt};
    \end{axis}
    \begin{axis}[
        at={(0.5\linewidth, 0)},
        width=0.48\linewidth,
        height=0.4\linewidth,
        xlabel={Parameter Index $\ell$},
        ylabel={Error $\abs{\yf(\pp_{\ell}) - \yr(\pp_{\ell})}$},
        grid=major,
        legend entries={RB, POD, $\Ltwo$-Opt-PSF},
        legend columns=2,
        legend style={
          at={(0.5, 1.05)},
          anchor=south,
          /tikz/every even column/.append style={column sep=2ex},
        },
        legend cell align={left},
        transpose legend,
        cycle list name=mplmarkonly,
        ymode=log,
        ymin=1e-6,
      ]
      \addplot table [x=p, y=rb] {thermalblock_output_errors.txt};
      \addplot table [x=p, y=pod] {thermalblock_output_errors.txt};
      \addplot table [x=p, y=dd] {thermalblock_output_errors.txt};
    \end{axis}
  \end{tikzpicture}
  \caption{Thermal block example \ac{fom} output and
    pointwise output errors for \acp{rom} of order $4$
    on the equidistant grid ${\mathtt{linspace}(0.1, 10, 5)}^4$ with
    $5^4 = 625$ points}%
  \label{fig:thermalblock}
\end{figure}
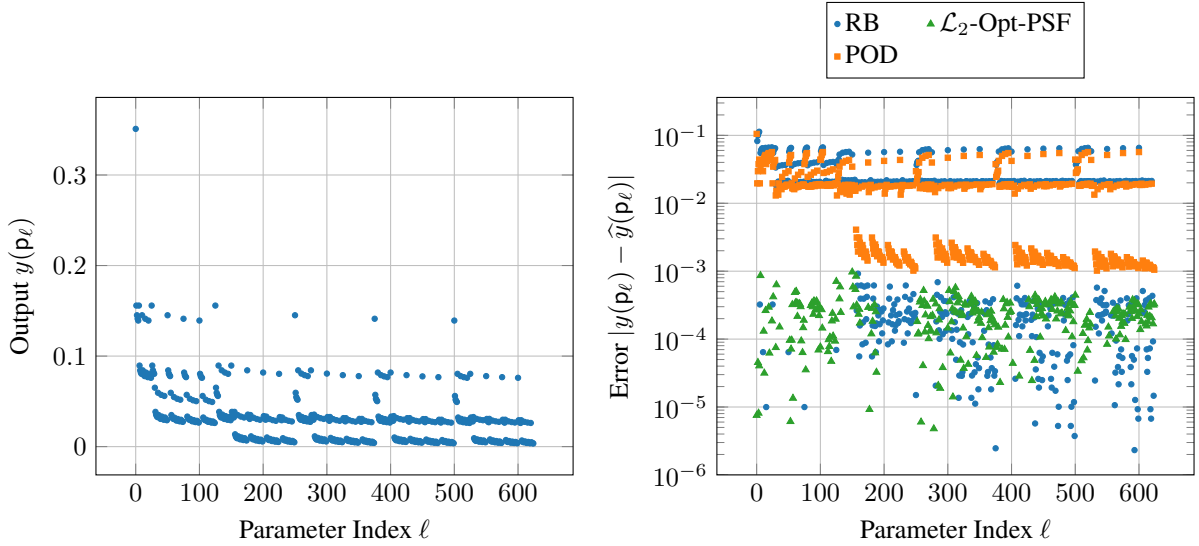

We look for structure-preserving \acp{ddrom} of the same form:
\begin{align*}
  \myparen*{
    \cAr_0
    + \pp_1 \cAr_1
    + \pp_2 \cAr_2
    + \pp_3 \cAr_3
    + \pp_4 \cAr_4
  }
  \xr(\pp)
  & = \cBr, \\*
  \yr(\pp)
  & = \cCr \xr(\pp).
\end{align*}
We choose $\nrom = 4$ as the reduced order.
We construct two projection-based \acp{rom} using \ac{rb} and \ac{pod}
trained on the grid
$\pset_{\textnormal{train}} = {\mathtt{linspace}(0.1, 10, 4)}^4$.
Then, using only the discrete output samples $y(\pp_\ell)$ from the same
training grid, i.e.,
a total of $N = 4^4 = 256$ output samples $y(\pp_\ell)$,
we run the discrete $\Ltwo$-Opt-PSF
(initialized with the \ac{pod} \ac{rom})
to minimize the \ac{mse}~\eqref{eq:mse}.
\Cref{fig:thermalblock} shows the outputs and the resulting errors over
$\pset_{\textnormal{test}}$.
The figure shows a significant reduction of the errors for the $\Ltwo$
\ac{ddrom}.
We note that $\pset_{\textnormal{test}}$ and $\pset_{\textnormal{train}}$
overlap only on $2^4 = 16$ points (corners) out of $5^4 = 625$ test parameter
values.
The relative $\Ltwo$ errors over $\pset_{\textnormal{test}}$ for \ac{rb},
\ac{pod}, and $\Ltwo$-Opt-PSF are, respectively,
$6.037 \times 10^{-1}$,
$4.8002 \times 10^{-1}$, and
$1.0266 \times 10^{-2}$.

\section{\texorpdfstring{$\Ltwo$}{L2}-optimal DDROMs and
  Projection-based ROMs}%
\label{sec:mor-by-projection}
\Ac{rb} and \ac{pod}, as implemented in~\Cref{sec:numerics-cont},
use Galerkin projection.
In \Cref{sec:numerics-cont},
we found that $\Ltwo$-optimal \acp{ddrom} obtained via \Cref{alg:l2opt} are not
based on Galerkin projection for the compliant Poisson examples.
The question remains whether the $\Ltwo$-optimal \acp{ddrom} can be constructed
by a (state-independent) Petrov-Galerkin projection.
The answer is clearly ``no'' when the \ac{fom} and \ac{ddrom} are of different
structure, in particular, when the \ac{ddrom} has an additional nonzero term as
we have in our extended $\Ltwo$-optimal \acp{ddrom}.
Therefore, we consider the case when the two models have the same
parameter-separable forms as in~\eqref{eq:rom-param-sep-form}.
Let $\Ar(\pp)$,  $\Br(\pp)$, and $\Cr(\pp)$ be the \ac{ddrom} quantities
obtained via \Cref{alg:l2opt}.
We want to see if there exist $V, W \in \RRfr$ of full-rank such that
\begin{equation*}
  \Ar(\pp) = W\tran \Af(\pp) V, \quad
  \Br(\pp) = W\tran \Bf(\pp), \quad
  \Cr(\pp) = \Cf(\pp) V.
\end{equation*}
Assuming that the sets of scalar functions
${\{\car_i\}}_{i = 1}^{\qAr}$,
${\{\cbr_j\}}_{j = 1}^{\qBr}$, and
${\{\ccr_k\}}_{k = 1}^{\qCr}$
are linearly independent,
this is equivalent to
\begin{subequations}\label{eq:projection}
  \begin{align}
    \label{eq:projection-A}
    \cAr_i & = W\tran \cAf_i V, & i = 1, 2, \ldots, \qAr, \\
    \label{eq:projection-B}
    \cBr_j & = W\tran \cBf_j, & j = 1, 2, \ldots, \qBr, \\
    \label{eq:projection-C}
    \cCr_k & = \cCf_k V, & k = 1, 2, \ldots, \qCr.
  \end{align}
\end{subequations}
We observe that~\eqref{eq:projection} is a system of nonlinear equations,
generally difficult to analyze.
However, the equations containing $\cBf_j$ and $\cCf_k$ are linear,
which we exploit next.

By vertically stacking~\eqref{eq:projection-C},
we obtain the linear system $\cCfblock V = \cCrblock$, where
\begin{equation}\label{eq:cC}
  \cCfblock =
  \begin{bmatrix}
    \cCf_1\tran &
    \cdots &
    \cCf_{\qCr}\tran
  \end{bmatrix}\tran
  \quad \textnormal{ and } \quad
  \cCrblock =
  \begin{bmatrix}
    \cCr_1\tran &
    \cdots &
    \cCr_{\qCr}\tran
  \end{bmatrix}\tran.
\end{equation}
If $\cCfblock$ is of full row rank (in particular, $\qCr \nout \le \nfom$),
using its singular value decomposition
\begin{equation}
  \label{eq:cC-svd}
  \cCfblock =
  U_{\cCfblock}
  \begin{bmatrix}
    \Sigma_{\cCfblock} & 0
  \end{bmatrix}
  \begin{bmatrix}
    V_{\cCfblock, 1} & V_{\cCfblock, 2}
  \end{bmatrix}\tran,
\end{equation}
we can write the solution to $\cCfblock V = \cCrblock$ as
\begin{equation}\label{eq:V-lstsq}
  V =
  {\underbrace{V_{\cCfblock, 1} \Sigma_{\cCfblock}^{-1} U_{\cCfblock}\tran
      \cCrblock}_{\eqqcolon V_1}}
  + V_{\cCfblock, 2} X,
\end{equation}
for some $X \in \RR^{(\nfom - \qCr \nout) \times \nrom}$.
Similarly, we find that
\begin{equation}\label{eq:W-lstsq}
  W =
  {\underbrace{U_{\cBfblock, 1} \Sigma_{\cBfblock}^{-1} V_{\cBfblock}\tran
      \cBrblock\tran}_{\eqqcolon W_1}}
  + U_{\cBfblock, 2} Y,
\end{equation}
for some $Y \in \RR^{(\nfom - \qBr \nin) \times \nrom}$,
where
\begin{gather}
  \label{eq:cB}
  \cBfblock =
  \begin{bmatrix}
    \cBf_1 & \cdots & \cBf_{\qBr}
  \end{bmatrix}\!, \quad
  \cBrblock =
  \begin{bmatrix}
    \cBr_1 & \cdots & \cBr_{\qBr}
  \end{bmatrix}\!, \\*
  \label{eq:cB-svd}
  \cBfblock =
  \begin{bmatrix}
    U_{\cBfblock, 1} \\
    U_{\cBfblock, 2}
  \end{bmatrix}\tran
  \begin{bmatrix}
    \Sigma_{\cBfblock} \\
    0
  \end{bmatrix}
  V_{\cBfblock}\tran,
\end{gather}
assuming that $\cBfblock$ is of full column rank.
Then, from~\eqref{eq:projection-A}, we obtain
\begin{equation}\label{eq:projection-A-X}
  (W_1 + U_{\cBfblock, 2} Y)\tran \cAf_i V_{\cCfblock, 2} X
  = \cAr_i - (W_1 + U_{\cBfblock, 2} Y)\tran \cAf_i V_1, \quad
  i = 1, 2, \ldots, \qAr.
\end{equation}
Stacking all the quantities, we obtain
\begin{align*}
  \begin{bmatrix}
    (W_1 + U_{\cBfblock, 2} Y)\tran \cAf_1 V_{\cCfblock, 2} \\
    \vdots \\
    (W_1 + U_{\cBfblock, 2} Y)\tran \cAf_{\qAr} V_{\cCfblock, 2}
  \end{bmatrix}
  X =
  \begin{bmatrix}
    \cAr_1 - (W_1 + U_{\cBfblock, 2} Y)\tran \cAf_1 V_1 \\
    \vdots \\
    \cAr_{\qAr} - (W_1 + U_{\cBfblock, 2} Y)\tran \cAf_{\qAr} V_1
  \end{bmatrix}\!,
\end{align*}
which is a system of $\nrom^2 \qAr$ equations and $\nrom (\nfom - \qCr \nout)$
unknowns.
Assuming that the system matrix is of full row rank,
there is at least one $X$ that solves the system.
Thus we have just proved the following result.
\begin{theorem}\label{thm:projection-rom}
  Let $\cBfblock$ in~\eqref{eq:cB} be of full column rank and
  $\cCfblock$ in~\eqref{eq:cC} of full row rank.
  Let their singular value decompositions be given by~\eqref{eq:cB-svd}
  and~\eqref{eq:cC-svd}.
  Also, let $V_1$ and $W_2$ be as in~\eqref{eq:V-lstsq} and~\eqref{eq:W-lstsq}.
  If there exists $Y \in \RR^{(\nfom - \qBr \nin) \times \nrom}$ such that
  \begin{align*}
    \begin{bmatrix}
      (W_1 + U_{\cBfblock, 2} Y)\tran \cAf_1 V_{\cCfblock, 2} \\
      \vdots \\
      (W_1 + U_{\cBfblock, 2} Y)\tran \cAf_{\qAr} V_{\cCfblock, 2}
    \end{bmatrix}
    \in \RR^{\qAr \nrom \times (\nfom - \qCr \nout)}
  \end{align*}
  is of full row rank,
  then there exist $V, W \in \RRfr$ satisfying~\eqref{eq:projection}.
\end{theorem}
This results says that, in the generic case,
all \acp{ddrom} with the same parameter-separable form as the \ac{fom} can be
formed using Petrov-Galerkin projection,
including $\Ltwo$-optimal ones.
Note that a similar, dual result to \Cref{thm:projection-rom} can be obtained by
fixing $X$ and solving for $Y$ in~\eqref{eq:projection-A-X}.

\section{Conclusions}%
\label{sec:conclusion}
We presented a gradient-based descent algorithm to construct data-driven
$\Ltwo$-optimal reduced-order models
that only requires access to output samples.
By appropriately defining the measure and parameter space, the framework we
developed covers both continuous (Lebesgue) and discrete cost functions, and
stationary and dynamical systems.
The various numerical examples illustrated the efficiency of the proposed
$\Ltwo$-optimal modeling approach.
Moreover, we have developed the generic conditions for \iac{ddrom} to be
projection-based.
The gradients derived in this paper have direct implications for and connections
to interpolatory model reduction methods and these issues will be revisited in a
separate work.

\appendix
\section{Proof of
  \texorpdfstring{\Cref{lem:rom-set-open-yr-bounded}}{Lemma}}%
\label{sec:proofs}
  First, we show that $\romset$ is an open subset of $R$.
  Let $(\cAr_i, \cBr_j, \cCr_k) \in \romset$ be arbitrary.
  The definition of $\romset$~\eqref{eq:setR} yields
  \begin{equation}\label{eq:Ar-M-bound}
    \abs*{\car_i(\pp)} \normF*{\Ar(\pp)^{-1}} \le M,
    \textnormal{ for all } i = 1, 2, \ldots, \qAr
    \textnormal{ and for } \measure\textnormal{-almost all } \pp \in \pset,
  \end{equation}
  for some $M > 0$.
  Then let $\Delta\cAr_i \in \RRrr$,
  for $i = 1, 2, \ldots, \qAr$,
  with
  $\sum_{\bar{\imath} = 1}^{\qAr} \normF{\Delta\cAr_{\bar{\imath}}}
  \le \frac{1}{2 M}$,
  be arbitrary.
  For $\measure$-almost all $\pp \in \pset$,
  using~\eqref{eq:Ar-M-bound} in the second inequality, we obtain
  \begin{equation}\label{eq:part-bound}
    \normF*{\sum_{\bar{\imath} = 1}^{\qAr}
      \car_{\bar{\imath}}(\pp) \Delta\cAr_{\bar{\imath}} \Ar(\pp)^{-1}}
    \le
    \sum_{\bar{\imath} = 1}^{\qAr}
    \abs{\car_{\bar{\imath}}(\pp)}
    \normF*{\Ar(\pp)^{-1}}
    \normF*{\Delta\cAr_{\bar{\imath}}}
    \le
    M
    \sum_{\bar{\imath} = 1}^{\qAr}
    \normF*{\Delta\cAr_{\bar{\imath}}}
    \le
    \frac{1}{2}.
  \end{equation}
  In the following,
  our goal is to show that
  $(\cAr_i + \Delta \cAr_i, \cBr_j, \cCr_k) \in \romset$.
  To start, we have
  \begin{align*}
    \abs*{\car_i(\pp)}
    &
      \normF*{\myparen*{\sum_{\bar{\imath} = 1}^{\qAr} \car_{\bar{\imath}}(\pp)
        \myparen*{\cAr_{\bar{\imath}} + \Delta\cAr_{\bar{\imath}}}}^{-1}}
      =
      \abs*{\car_i(\pp)}
      \normF*{\myparen*{\Ar(\pp)
        + \sum_{\bar{\imath} = 1}^{\qAr}
          \car_{\bar{\imath}}(\pp)
          \Delta\cAr_{\bar{\imath}}}^{-1}} \\
    & =
      \abs*{\car_i(\pp)}
      \normF*{
        \Ar(\pp)^{-1}
        \myparen*{I
          + \sum_{\bar{\imath} = 1}^{\qAr}
            \car_{\bar{\imath}}(\pp)
            \Delta\cAr_{\bar{\imath}}
            \Ar(\pp)^{-1}
        }^{-1}
      } \\
    & \le
      \abs*{\car_i(\pp)}
      \normF*{\Ar(\pp)^{-1}}
      \normF*{
        \myparen*{I
          + \sum_{\bar{\imath} = 1}^{\qAr}
            \car_{\bar{\imath}}(\pp)
            \Delta\cAr_{\bar{\imath}}
            \Ar(\pp)^{-1}
        }^{-1}
      }.
  \end{align*}
  Using that $\normF{(I - X)^{-1}} \le \frac{1}{1 - \normF{X}}$
  for all $X \in \CCrr$ such that $\normF{X} < 1$,
  from~\eqref{eq:part-bound} and~\eqref{eq:Ar-M-bound} we get
  \begin{equation*}
    \abs*{\car_i(\pp)}
    \normF*{\myparen*{\sum_{\bar{\imath} = 1}^{\qAr} \car_{\bar{\imath}}(\pp)
        \myparen*{\cAr_{\bar{\imath}} + \Delta\cAr_{\bar{\imath}}}}^{-1}}
    \le
    \abs*{\car_i(\pp)}
    \normF*{\Ar(\pp)^{-1}}
    \frac{1}{1 - \frac{1}{2}}
    \le 2 M
    < \infty,
  \end{equation*}
  for $\measure$-almost every $\pp \in \pset$.
  Therefore, we have
  $(\cAr_i + \Delta \cAr_i, \cBr_j, \cCr_k) \in \romset$.
  Since $\cBr_j$ and $\cCr_k$ are arbitrary, it follows that there is an open
  neighborhood of $(\cAr_i, \cBr_j, \Cr_k)$ in $\romset$.

  Next we prove that $\normLtwomu{\yr} < \infty$.
  Note that from~\eqref{eq:Ar-M-bound},
  if $\car_i(\pp) \neq 0$,
  then $\normF{\Ar(\pp)^{-1}} \le \frac{M}{\abs{\car_i(\pp)}}$.
  Furthermore,
  note that the set of parameter values $\pp \in \pset$ such that
  $\car_i(\pp) = 0$ for all $i = 1, 2, \ldots, \qAr$
  forms a set of $\measure$-measure zero
  (otherwise, this would contradict~\eqref{eq:abc-l2-bounded}).
  Therefore,
  \begin{equation}\label{eq:Ar-bound}
    \normF*{\Ar(\pp)^{-1}}
    \le \min_{i} \frac{M}{\abs{\car_i(\pp)}}
    \le \frac{M}{\max_{i} \abs{\car_i(\pp)}},
    \textnormal{ for } \measure\textnormal{-almost all } \pp \in \pset.
  \end{equation}
  Using submultiplicativity and the triangle inequality, we obtain
  \begin{align}
    \nonumber
    \normLtwo{\yr}^2
    & =
      \int_{\pset} \normF*{\Cr(\pp) \Ar(\pp)^{-1} \Br(\pp)}^2 \difm{\pp} \\
    \label{eq:l2-bound-norm-prod}
    & \le
      \int_{\pset}
      \normF*{\Cr(\pp)}^2 \normF*{\Ar(\pp)^{-1}}^2 \normF*{\Br(\pp)}^2
      \difm{\pp} \\
    \nonumber
    & \le
      \int_{\pset}
      \myparen*{\sum_{k = 1}^{\qCr} \abs*{\ccr_k(\pp)} \normF*{\cCr_k}}^2
      \normF*{\Ar(\pp)^{-1}}^2
      \myparen*{\sum_{j = 1}^{\qBr} \abs*{\cbr_j(\pp)} \normF*{\cBr_j}}^2
      \difm{\pp}.
  \end{align}
  Using $\normF{\cBr_j} \le \max_{\bar{\jmath}} \normF{\cBr_{\bar{\jmath}}}$ and
  $\normF{\cCr_k} \le \max_{\bar{k}} \normF{\cCr_{\bar{k}}}$,
  we find
  \begin{align*}
    \normLtwo{\yr}^2
    & \le
      \max_{\bar{\jmath}} \normF*{\cBr_{\bar{\jmath}}}^2
      \max_{\bar{k}} \normF*{\cCr_{\bar{k}}}^2
      \int_{\pset}
      \myparen*{\sum_{j = 1}^{\qBr} \abs*{\cbr_j(\pp)}}^2
      \myparen*{\sum_{k = 1}^{\qCr} \abs*{\ccr_k(\pp)}}^2
      \normF*{\Ar(\pp)^{-1}}^2
      \difm{\pp}.
  \end{align*}
  Next, using~\eqref{eq:Ar-bound}, we obtain
  \begin{equation*}
    \normLtwo{\yr}^2
    \le
      M^2
      \max_{\bar{\jmath}} \normF*{\cBr_{\bar{\jmath}}}^2
      \max_{\bar{k}} \normF*{\cCr_{\bar{k}}}^2
      \int_{\pset}
      \frac{
        \myparen*{\sum_{j = 1}^{\qBr} \abs*{\cbr_j(\pp)}}^2
        \myparen*{\sum_{k = 1}^{\qCr} \abs*{\ccr_k(\pp)}}^2
      }{\max_{i} \abs{\car_i(\pp)}^2}
      \difm{\pp}.
  \end{equation*}
  Finally, using that
  $\max_{i = 1, 2, \ldots, \nfom} x_i
  \ge \sqrt{\myparen{\sum_{i = 1}^{\nfom} x_i^2} / \nfom}$
  for nonnegative numbers $x_i$,
  \begin{align*}
    \normLtwo{\yr}^2
    & \le
      \qAr^2 M^2
      \max_{\bar{\jmath}, \bar{k}}
      \normF*{\cBr_{\bar{\jmath}}}^2
      \normF*{\cCr_{\bar{k}}}^2
      \int_{\pset}
      \frac{
        \myparen*{\sum_{j = 1}^{\qBr} \abs*{\cbr_j(\pp)}}^2
        \myparen*{\sum_{k = 1}^{\qCr} \abs*{\ccr_k(\pp)}}^2
      }{\myparen*{\sum_{i = 1}^{\qAr} \abs*{\car_i(\pp)}}^2}
      \difm{\pp} < \infty,
  \end{align*}
  which completes the proof.
\section{Proof of \texorpdfstring{\Cref{thm:gradients}}{Theorem}}%
\label{sec:proofstheorem}
  Rewrite the objective function as
  \begin{subequations}
    \begin{align}
      \obj = {}
      &
        \normLtwomu{\yf}^2
        \label{eq:J1} \\*
      & -
        2 \int_{\pset}
        \trace*{
        \yf(\pp)\herm
        \Cr(\pp) \Ar(\pp)^{-1} \Br(\pp)
        }
        \difm{\pp}
        \label{eq:J2} \\*
      & +
        \int_{\pset}
        \trace*{
        \Br(\pp)\herm \Ar(\pp)\mherm \Cr(\pp)\herm
        \Cr(\pp) \Ar(\pp)^{-1} \Br(\pp)
        }
        \difm{\pp},
        \label{eq:J3}
    \end{align}
  \end{subequations}
  where we used the fact that
  $\ipLtwomu{\yf}{\yr} = \ipLtwomu{\yr}{\yf} \in \RR$ (\Cref{assumption:fom}).
  The part of $\obj$ in~\eqref{eq:J1} does not depend on the reduced quantities,
  so it does not contribute to the gradient.
  Let $\obj_2$ denote the second term~\eqref{eq:J2} in the cost function $\obj$,
  i.e.,
  \[
    \obj_2 = -2 \int_{\pset}
    \trace*{
      \yf(\pp)\herm
      \Cr(\pp) \Ar(\pp)^{-1} \Br(\pp)
    }
    \difm{\pp}.
  \]
  We start by computing $\nabla_{\cAr_i} \obj_2$.
  To do so, we evaluate $\obj_2\myparen{\cAr_i + \Delta\cAr_i}$
  for a perturbation $\Delta\cAr_i$ to obtain
  \begin{align*}
    \obj_2\myparen*{\cAr_i + \Delta\cAr_i}
    & =
      -2 \int_{\pset}
      \trace*{
        \yf(\pp)\herm
        \Cr(\pp) \myparen*{\Ar(\pp) + \car_i(\pp) \Delta\cAr_i}^{-1} \Br(\pp)
      }
      \difm{\pp} \\
    & =
      -2 \int_{\pset}
      \trace*{
        \yf(\pp)\herm
        \Cr(\pp)
        \myparen*{I + \car_i(\pp) \Ar(\pp)^{-1} \Delta\cAr_i}^{-1}
        \Ar(\pp)^{-1}
        \Br(\pp)
      }
      \difm{\pp}.
  \end{align*}
  Assuming small enough $\Delta\cAr_i$,
  using the property in~\eqref{eq:setR} and
  applying the Neumann series formula yield
  \begin{align}
    \notag
    \obj_2\myparen*{\cAr_i + \Delta\cAr_i}
    & =
      -2 \int_{\pset}
      \trace*{
        \yf(\pp)\herm
        \Cr(\pp) \Ar(\pp)^{-1} \Br(\pp)
      }
      \difm{\pp} \\*
    \notag
    & \ \
      + 2 \int_{\pset}
      \trace*{
        \car_i(\pp)
        \yf(\pp)\herm
        \Cr(\pp)
        \Ar(\pp)^{-1} \Delta\cAr_i \Ar(\pp)^{-1}
        \Br(\pp)
      }
      \difm{\pp} \\*
    \notag
    & \ \
      - 2 \int_{\pset}
      \trace*{
        \yf(\pp)\herm
        \Cr(\pp)
        \sum_{m = 2}^\infty
          \myparen*{-\car_i(\pp) \Ar(\pp)^{-1} \Delta\cAr_i}^m \Ar(\pp)^{-1}
        \Br(\pp)
      }
      \difm{\pp} \\
    \notag
    & =
      \obj(\cAr_i)
      + \ipF*{2 \int_{\pset} \overline{\car_i(\pp)} \Ar(\pp)\mherm \Cr(\pp)\herm
        \yf(\pp) \Br(\pp)\herm \Ar(\pp)\mherm \difm{\pp}}%
      {\Delta\cAr_i} \\*
    \label{eq:nabla-J-Ar-remaining-terms}
    & \ \
      - 2
      \sum_{m = 2}^\infty
      \int_{\pset}
      \trace*{
        \yf(\pp)\herm
        \Cr(\pp)
        \myparen*{-\car_i(\pp) \Ar(\pp)^{-1} \Delta\cAr_i}^m \Ar(\pp)^{-1}
        \Br(\pp)
      }
      \difm{\pp}.
  \end{align}
  First, we check that the candidate for the gradient,
  resulting from the second term in the last equation,
  is indeed bounded:
  \begin{align*}
    &
      \normF*{
        \int_{\pset} \car_i(\overline{\pp})
        \Ar(\pp)\mherm \Cr(\pp)\herm
        \yf(\pp)
        \Br(\pp)\herm \Ar(\pp)\mherm
        \difm{\pp}
      } \\
    & \le
      \int_{\pset} \abs*{\car_i(\pp)}
      \normF*{\Ar(\pp)^{-1}}
      \normF*{\Cr(\pp)} \normF*{\Ar(\pp)^{-1}} \normF*{\Br(\pp)}
      \normF*{\yf(\pp)}
      \difm{\pp} \\
    & \le
      \normLinfmu*{\car_i(\dotvar) {\Ar(\dotvar)}^{-1}}
      \int_{\pset}
      \normF*{\Cr(\pp)} \normF*{\Ar(\pp)^{-1}} \normF*{\Br(\pp)}
      \normF*{\yf(\pp)}
      \difm{\pp} < \infty,
  \end{align*}
  where we used~\eqref{eq:abc-l2-bounded} and~\eqref{eq:setR}
  (see~\eqref{eq:l2-bound-norm-prod} in the proof of
  \Cref{lem:rom-set-open-yr-bounded}).
  Second, we check that the remaining terms
  in~\eqref{eq:nabla-J-Ar-remaining-terms} are of lower order:
  \begin{align*}
    &
      \abs*{
        \sum_{m = 2}^\infty
        \int_{\pset}
        \trace*{
          \yf(\pp)\herm
          \Cr(\pp)
          \myparen*{-\car_i(\pp) \Ar(\pp)^{-1} \Delta\cAr_i}^m \Ar(\pp)^{-1}
          \Br(\pp)
        }
        \difm{\pp}
      } \\
    & \le
      \normLtwomu{\yf}
      \sum_{m = 2}^\infty
      \normLtwomu*{
        \Cr(\dotvar)
        \myparen*{-\car_i(\dotvar) {\Ar(\dotvar)}^{-1} \Delta\cAr_i}^m
          {\Ar(\dotvar)}^{-1}
        \Br(\dotvar)
      } \\
    & \le
      \normLtwomu{\yf}
      \sum_{m = 2}^\infty
      \normLinf*{\car_i(\dotvar) {\Ar(\dotvar)}^{-1}}^m
      \normLtwomu*{
        \normF[\big]{\Cr(\dotvar)}
        \normF[\big]{{\Ar(\dotvar)}^{-1}}
        \normF[\big]{\Br(\dotvar)}
      }
      \normF[\big]{\Delta\cAr_i}^m \\
    & =
      \normLtwomu{\yf}
      \normLtwomu*{
        \normF[\big]{\Cr(\dotvar)}
        \normF[\big]{{\Ar(\dotvar)}^{-1}}
        \normF[\big]{\Br(\dotvar)}
      }
      \frac{
        \normLinf*{\car_i(\dotvar) {\Ar(\dotvar)}^{-1}}^2
        \normF[\big]{\Delta\cAr_i}^2
      }{1 -
        \normLinf*{\car_i(\dotvar) {\Ar(\dotvar)}^{-1}}
        \normF[\big]{\Delta\cAr_i}
      } \\
    & =
      o\myparen*{\normF[\big]{\Delta\cAr_i}}.
  \end{align*}
  Therefore,
  \begin{align*}
    \nabla_{\cAr_i} \obj_2
    & =
      2 \int_{\pset} \car_i(\overline{\pp})
      \Ar(\pp)\mherm \Cr(\pp)\herm
      \yf(\pp)
      \Br(\pp)\herm \Ar(\pp)\mherm
      \difm{\pp}.
  \end{align*}
  Next, we compute $\nabla_{\cBr_j} \obj_2$ similarly by evaluating
  $\obj_2\myparen{\cBr_j + \Delta\cBr_j}$:
  \begin{align*}
    \obj_2\myparen*{\cBr_j + \Delta\cBr_j} = {}
    &
      {-2} \int_{\pset}
      \trace*{
        \yf(\pp)\herm
        \Cr(\pp) \Ar(\pp)^{-1} \myparen*{\Br(\pp) + \cbr_j(\pp) \Delta\cBr_j}
      }
      \difm{\pp} \\
    = {}
    &
      \obj_2(\cBr_j)
      - 2 \int_{\pset}
      \trace*{
        \cbr_j(\pp)
        \yf(\pp)\herm
        \Cr(\pp) \Ar(\pp)^{-1} \Delta\cBr_j
      }
      \difm{\pp} \\
    = {}
    &
      \obj_2(\cBr_j)
      - 2
      \ipF*{
        \int_{\pset}
        \cbr_j(\overline{\pp})
        \Ar(\pp)\mherm \Cr(\pp)\herm \yf(\pp)
        \difm{\pp}
      }{\Delta\cBr_j}.
  \end{align*}
  It follows from~\eqref{eq:abc-l2-bounded} and~\eqref{eq:setR} that the
  mapping $\pp \mapsto \cbr_j(\overline{\pp}) \Ar(\pp)\mherm \Cr(\pp)\herm$
  is square-integrable.
  Therefore
  \begin{equation*}
    \nabla_{\cBr_j} \obj_2
    =
    -2
    \int_{\pset}
    \cbr_j(\overline{\pp})
    \Ar(\pp)\mherm \Cr(\pp)\herm \yf(\pp)
    \difm{\pp}.
  \end{equation*}
  Similarly, one can obtain
  \begin{equation*}
    \nabla_{\cCr_k} \obj_2
    =
    -2
    \int_{\pset}
    \ccr_k(\overline{\pp})
    \yf(\pp) \Br(\pp)\herm \Ar(\pp)\mherm
    \difm{\pp}.
  \end{equation*}
  Finally, after differentiating the last part of $\obj$ in~\eqref{eq:J3},
  we obtain
  \begin{align*}
    \nabla_{\cAr_i} \obj = {}
    &
      2 \int_{\pset} \car_i(\overline{\pp})
      \Ar(\pp)\mherm \Cr(\pp)\herm
      \myparen*{\yf(\pp) - \Cr(\pp) \Ar(\pp)^{-1} \Br(\pp)}
      \Br(\pp)\herm \Ar(\pp)\mherm
      \difm{\pp}, \\
    \nabla_{\cBr_j} \obj = {}
    &
      2 \int_{\pset} \cbr_j(\overline{\pp})
      \Ar(\pp)\mherm \Cr(\pp)\herm
      \myparen*{\Cr(\pp) \Ar(\pp)^{-1} \Br(\pp) - \yf(\pp)}
      \difm{\pp},\ \textnormal{ and} \\
    \nabla_{\cCr_k} \obj = {}
    &
      2 \int_{\pset} \ccr_k(\overline{\pp})
      \myparen*{\Cr(\pp) \Ar(\pp)^{-1} \Br(\pp) - \yf(\pp)}
      \Br(\pp)\herm \Ar(\pp)\mherm
      \difm{\pp},
  \end{align*}
  which completes the proof.


\newcommand{\etalchar}[1]{$^{#1}$}

\addcontentsline{toc}{chapter}{References}
\end{document}